\documentclass[12pt]{article}
\usepackage[utf8]{inputenc}
\usepackage{lmodern}
\usepackage{graphicx, xcolor, mathrsfs}
\usepackage{amsfonts}
\usepackage{stmaryrd}
\usepackage{blkarray, bigstrut}
\usepackage{mathtools}
\usepackage[nice]{nicefrac}
\usepackage{amsmath,amsthm,amssymb}
\usepackage[mathlines]{lineno}
\usepackage{mathabx}
\usepackage{accents}

\usepackage{enumerate}
\usepackage{enumitem}
\usepackage[colorlinks=true,citecolor=black,linkcolor=black,urlcolor=blue]{hyperref}
\usepackage[top=.9in, bottom=.9in, left=.5in , right=.5in]{geometry}
\usepackage{caption}
\usepackage{xifthen} 
\usepackage{wrapfig}
\usepackage[numbers, square]{natbib}
\usepackage{changes}
\usepackage{comment}
\usepackage{float}
\usepackage{ifthen}
\mathtoolsset{showonlyrefs}
\usepackage{tikz}
\usepackage{caption}                       
\usetikzlibrary{arrows}
\usetikzlibrary{graphs,graphs.standard}
\usetikzlibrary{positioning,arrows.meta}
\usetikzlibrary{shapes.multipart}
\usetikzlibrary{arrows}
\usetikzlibrary{automata}
\definecolor{arrowblue}{RGB}{0,0,0}  

\usepackage[splitrule]{footmisc} 
\usepackage{caption}
\newtheorem{thm}{Theorem}[section]
\newtheorem*{thm*}{Theorem}
\newtheorem{lem}[thm]{Lemma}
\newtheorem{clm}{Claim}[thm]
\newtheorem{prop}[thm]{Proposition}
\newtheorem{cor}[thm]{Corollary}
\newtheorem{assumption}{Assumption}[section]

\theoremstyle{definition}
\newtheorem{rmq}[thm]{Remark}

\newtheorem{defn}[thm]{Definition}

\newcommand{\E}[2][]{\ensuremath{\mathbb{E}_{#1}\left[#2 \right]}}
\newcommand{\Prob}[2][]{\ensuremath{\mathbb{P}_{#1} \left(#2 \right)}}
\def\Var{\mathop{\rm Var}\nolimits}
\newcommand{\var}[2][]{\ensuremath{\Var_{#1} \left(#2 \right)}}
\newcommand{\eps}{\varepsilon}
\newcommand{\dd}{\mathrm d}
\DeclareMathOperator{\exponentialrv}{Exp}
\DeclareMathOperator{\geometricrv}{Geom}

\DeclareMathOperator{\mon}{Mon}
\DeclareMathOperator{\lead}{Lead}
\DeclareMathOperator{\slead}{SLead}
\DeclareMathOperator{\win}{Win}
\DeclareMathOperator{\disc}{disc.}
\DeclareMathOperator{\cont}{cont.}
\newcommand{\atomicset}[1]{\mathcal{A}_{#1}}
\newcommand{\atomicpart}[1]{\mu_{\disc}^{#1}}
\newcommand{\contpart}[1]{\mu_{\cont}^{#1}}
\newcommand{\meas}[1]{\mu^{#1}}

\newcommand{\Exp}[1]{\exponentialrv\left( #1 \right)}
\newcommand{\Geom}[1]{\geometricrv\left( #1 \right)}

\newcommand{\epstwo}{\eta}
\newcommand{\ssup}[1] {{\scriptscriptstyle{({#1}})}}

\title{Fixation of leadership in non-Markovian growth processes}
\author{T. Iyer\footnote{Weierstrass Institute for Applied Analysis and Stochastics, Mohrenstrasse 39, 10117 Berlin, Germany.}}
\date{August 21, 2024}

\begin{document}
\maketitle
\abstract{Consider a model where $N$ equal agents possess `values', belonging to $\mathbb{N}_0$, that are subject to incremental growth over time. More precisely, the values of the agents are represented by $N$ independent, increasing $\mathbb{N}_0$ valued processes with random, independent waiting times between jumps. We show that the event that a single agent possesses the maximum value for all sufficiently large values of time (called `leadership') occurs with probability zero or one, and provide necessary and sufficient conditions for this to occur. Under mild conditions we also provide criteria for a single agent to become the unique agent of maximum value for all sufficiently large times, and also conditions for the emergence of a unique agent having value that tends to infinity before `explosion' occurs (i.e. conditions for `strict leadership' or `monopoly' to occur almost surely). The novelty of this model lies in allowing non-exponentially distributed waiting times between jumps in value. In the particular case when waiting times are mixtures of exponential distributions, we improve a well-established result on the `balls in bins' model with feedback, removing the requirement that the feedback function be bounded from below and also allowing random feedback functions. As part of the proofs we derive necessary and sufficient conditions for the distribution of a convergent series of independent random variables to have an atom on the real line, a result which we believe may be of interest in its own right.}
\section{Introduction}
\subsection{Background and motivation}
In various circumstances, it is of interest to model the dynamics of values of various agents subject to growth over time. As the values of agents grow it may be the case that they can leverage their value to grow at a faster rate, leading to reinforced growth and `rich-gets-richer' effects. In such a situation a natural question is whether or not a `leader' forms: a single agent that eventually obtains and retains maximal value for all sufficiently large times. 

Such a model is ubiquitous. One of the first, well-known, applications was to economics, where one might consider agents to be companies, with `values' representing wealth~\cite{path-dependence-economy, strong-laws-path-dependent}. In this case, `leadership' may indicate the formation of a company with dominant market share, or even the formation of `monopolies'~\cite{balls-in-bins-feedback-02}. In other applications, one may consider agents representing political policies, with `leadership' representing the formation of a most popular policy (called institutional `stickiness' in~\cite{politics}). In modelling neuron development, agents might also represent neurites, before the formation of a `leader', that is, the specification of an axon~\cite{khanin-neuron-polarity}. 

In previous applications, these dynamics have been modelled by `balls-in-bins' processes~\cite{balls-in-bins-feedback-02} (also known as non-linear urns~\cite{grosskinksy-generalised-urns-mono} or, originally, as generalised urn processes~\cite{strong-law-hill-sudderth, generalized-2}). In these processes, agents are represented by a fixed number of \emph{urns}, and the values of agents, encoded by the number of \emph{balls} the associated urn contains. The rate of `growth' is measured by a feedback function $f\colon\mathbb{N}_{0} \rightarrow (0,\infty)$. Then starting from a given initial condition, at each discrete time-step, one selects an urn containing $j$ balls (say) with probability proportional to $f(j)$, and adds a ball to it. 

This model generalises the classical P\'olya-Eggenberger urn scheme~\cite{polya-urn-original} -- the case $f(j) = j$ -- and thus was first introduced under the name `generalised urns' in~\cite{strong-law-hill-sudderth}. It was later re-introduced under the name `balls-in-bins' in~\cite{balls-in-bins-feedback-02}, however, we use the latter terminology to avoid confusion with other generalisations of the P\'olya-Eggenberger urn scheme (e.g.~\cite{janson-urns}). 

In these processes, \emph{strict leadership} (known as `eventual leadership' in~\cite{oliveira-brownian-motion}) occurs if there is a unique urn containing the maximum number of balls for all but finitely many time-steps. On the other hand \emph{monopoly} (a term coined in~\cite{balls-in-bins-feedback-02}), occurs if for all but finitely many time-steps, a single urn is chosen for new additions of balls. A well-known result, as stated in~\cite{oliveira-brownian-motion}, is the following:

\begin{thm*}[From~\cite{balls-in-bins-feedback-02, khanin-neuron-polarity, scaling-result-explosive-processes, oliveira-onset-of-dominance,oliveira-thesis, oliveira-brownian-motion}]
    In the balls-in-bins process, with feedback function $f\colon\mathbb{N}_{0} \rightarrow (0, \infty)$, regardless of the initial conditions of the process:
    \begin{itemize}
        \item Monopoly occurs with probability zero or one. Moreover, monopoly occurs with probability one if and only if $\sum_{j=1}^{\infty} 1/f(j) < \infty$. 
        \item If the feedback function $f$ is bounded from below, strict leadership occurs with probability zero or one. Moreover, strict leadership occurs with probability one if and only if $\sum_{j=1}^{\infty} 1/f(j)^2 < \infty$. 
    \end{itemize}
\end{thm*}
The first result concerning monopoly comes from Rubin's argument (presented in~\cite{davis-rubin}). On the other hand, the result regarding strict leadership was first proved in the case $f(j) = (j+1)^{p}$ in~\cite{khanin-neuron-polarity}, and generalised by subsequent works of Oliveira and Spencer~\cite{oliveira-thesis, oliveira-brownian-motion, oliveira-2005-avoiding-defeat-balls-in-bins-process}. Thus, if $f(j) = (j+1)^{p}$, for $p \in (0, \infty)$ - corresponding to instance of the model studied in, for example,~\cite{chung-generalisations-of-polya-urn, balls-in-bins-feedback-02, khanin-neuron-polarity} - the model exhibits \emph{phase-transitions} at the values $p = 1$ and $p = 1/2$ respectively, corresponding to whether or not monopoly, or strict leadership occur. 

It is important to note that the conditions on summability of $\sum_{j=0}^{\infty} 1/f(j)$ and $\sum_{j=0}^{\infty} 1/f(j)^2$ appear more widely in other models of reinforcement, where $f$, informally,  represents the degree of reinforcement. The condition of summability of $\sum_{j=0}^{\infty} 1/f(j)$, appears as a condition for `explosion' in first-passage percolation in trees~\cite{explosion-first-passage}, in `fixation' of edge-reinforced random walks~\cite{davis-rubin, sellke-reinforced-random-walk, pemantle-review}, and in `connectivity transitions' of growing trees of the `generalised preferential attachment' type~\cite{oliveira-spencer-conn-transitions, holmgren-janson, inhom-sup-pref-attach}. On the other hand, the condition of summability of $\sum_{j=0}^{\infty} 1/f(j)^2$ also arise in the context of growing generalised preferential attachment graphs; in criteria for the emergence of the emergence of a `persistent hub'. In these latter models, a persistent hub represents a node in the graph whose degree is the largest for all but finitely many time-steps in the evolution of the network~\cite{dereich-morters-persistence, galashin, banerjee-bhamidi}. 

A common approach to analyse the balls-in-bins model, and  other similar models is via a continuous time embedding. This approach dates back to Athreya and Karlin in a related model~\cite{arthreya-karlin-embedding-68}, but was first exploited in the context of the balls-in-bins model by Rubin~\cite{davis-rubin}. One considers the collection of urns as a continuous time Markov process, with the number of balls in a particular urn increasing from $j$ to $j+1$ independently at rate $f(j)$. The embedded Markov chain recovers the original balls-in-bins model. The continuous time representation immediately gives meaning to condition on the summability of $\sum_{j=0}^{\infty} 1/f(j)$, as a necessary and sufficient condition for explosion of the pure-birth process associated with the number of balls in a particular urn. 

There is a large literature related to the balls-in-bins model, concerning properties such as the limiting proportion of balls in certain urn~\cite{strong-law-hill-sudderth, generalized-2, strong-laws-path-dependent, grosskinksy-generalised-urns-mono, grosskinsky-wages-capital-returns-generalized}, the emergence of `weak monopoly', that is, an urn whose limiting `market share' tends to one~\cite{grosskinksy-generalised-urns-mono}, and the number of balls of `losing type' when monopoly occurs~\cite{grosskinsky-tails-explosive-birth-processes}. Other results are concerned with properties such as the probability of monopoly, or leadership by an urn when one varies the initial conditions of the urns~\cite{scaling-result-explosive-processes, oliveira-2005-avoiding-defeat-balls-in-bins-process, oliveira-brownian-motion, oliveira-thesis}. A number of generalisations of the model have also been studied, including, models with varied feedback functions across different urns~\cite{grosskinksy-generalised-urns-mono, balls-in-bins-asymmetric-feedback}, models where the number of balls replaced is asymmetric~\cite{grosskinsky-wages-capital-returns-generalized}, time-dependent models, with random numbers of balls added at each time-step~\cite{pemantle-time-dependent, sidorova-2018-timedependent-balls-bins-model}, and interacting urn models~\cite{interacting-urns-nonlinear-types, launay-2012-interacting-urn-models, qin-2023-interacting-urn-models-strong}. We also remark that there is a large literature on urns that generalise the P\'olya-Eggenberger urn scheme in other ways, for example~\cite{arthreya-karlin-embedding-68, janson-urns, infinite-colour-urns, infinite-colour-urns-2}. 

However, one limitation of the balls-in-bins model, and its extensions in the literature is that they are Markov processes. Moreover, in many contexts, one might regard the continuous time embedding of this model as more realistic in modelling the growth of values of agents: if values represent the wealth of companies, the times when `growth' take place often occurs in continuous time, rather than in discrete time-steps. However, in these applications it is not desirable to require exponentially distributed waiting times between jumps, as this may not always be realistic. 
This motivates the present study: the analysis of more general non-Markovian processes which generalise the balls-in-bins model, where the waiting times in the underlying continuous time process are independent but not necessarily exponentially distributed. 

\subsection{Contributions of this paper and overview}
The contributions of this paper are the following: 
\begin{itemize}
    \item In Theorem~\ref{thm:leadership-classification} we derive necessary and sufficient conditions for leadership to occur with probability zero or one in these more general competing growth processes, in other words, conditions for a single agent to have maximum value for all but finitely many time-steps. Moreover, in Theorem~\ref{thm:leadership-classification} and Corollary~\ref{cor:strict-leadership} we provide sufficient criteria for monopoly and strict leadership to occur with probability zero or with probability one. Unlike leadership, the latter two properties are not zero-one events in general, as shown in Remark~\ref{rem:non-zero-one}.
    \item In Theorem~\ref{thm:no-atom} we prove an important auxiliary result which may be of independent interest:  necessary and sufficient criteria for the distribution of a convergent sum of independent random variables to have no atom on the real-line. Another auxiliary result, Proposition~\ref{prop:fluctuating-sum}, shows that when a series of independent random variables fails to converge, and the summands are symmetric, the partial sums associated with the series cross the origin infinitely often. 
    \item The results from Theorem~\ref{thm:leadership-classification} shed new light on the summability of $\sum_{j=0}^{\infty} 1/f(j)^2$ in regards to the balls-in-bins process scheme, showing that this condition arises when analysing the convergence of certain random series in Theorem~\ref{thm:leadership-classification}. In particular, applying Theorem~\ref{thm:leadership-classification} leads to a new result regarding strict leadership in balls-in-bins schemes in Theorem~\ref{thm:balls-in-bins},
    allowing for the feedback function to be random, and showing that it need not be bounded from below.
\end{itemize}
Section~\ref{section:desc} contains the formal description of the model and the main results related to the model: Theorem~\ref{thm:leadership-classification}, Corollary~\ref{cor:strict-leadership} and Proposition~\ref{prop:counter}. It also contains auxiliary results on random series that may be of independent interest: Theorem~\ref{thm:no-atom} and Proposition~\ref{prop:fluctuating-sum}. Section~\ref{sec:balls-in-bins} deals with applications to the `balls-in-bins' scheme, providing a formal overview of that model, and the main result in Theorem~\ref{thm:balls-in-bins}. Section~\ref{sec:proofs} includes the proofs of the results appearing in this article, with sub-sections that can generally (aside from a few global definitions) be read independently of each other. 


\subsection{Model description and main results} \label{section:desc}
We consider a finite family of $\mathbb{N}_{0}$ valued growth processes with independent waiting times between jumps. Suppose we have $A \geq 2$ \emph{agents} labelled by the elements of $[A] := \left\{1, \ldots, A\right\}$. To each $a \in [A]$, we associate an identically distributed sequence of mutually independent random variables $(X^{\ssup{a}}_{j})_{j \in \mathbb{N}}$, taking values in $[0, \infty)$. Each agent $a \in [A]$ has a \emph{value} $v_{a}\colon [0, \infty) \rightarrow \mathbb{N}$ such that, the value of $a$ at time $t$, $v_a(t)$, increases over time. The quantity $X^{\ssup{a}}_{k}$ is the time taken for the value of $a$ to increase from $k-1$ to $k$, and the agent $a$ begins with a value of $v_{a}(0)$. Thus, given the \emph{initial value} $v_{a}(0)$, for $k \in \mathbb{N}_{0}$ we have 
\[
v_a(t) = v_{a}(0) + k \quad \text{if and only if} \quad \sum_{j=v_{a}(0) + 1}^{v_{a}(0) + k} X^{\ssup{a}}_{j} \leq t < \sum_{j=v_{a}(0) + 1}^{v_{a}(0) + k+1} X^{\ssup{a}}_{j}.
\]  
We are interested in the vector of values of agents as time evolves, until, possibly, an agent reaches infinite value. To do so, we set $\tau_{0} := 0$. Then for each $n \in \mathbb{N}$, we define 
\begin{equation} \label{eq:stop-times}
\tau_{n} := \inf{\left\{t \geq 0\colon\sum_{a =1}^{A} v_{a}(t) \geq n \right\}}. 
\end{equation}
We are generally interested in the discrete process $(v_{a}(\tau_{n}))_{a \in [A], n \in \mathbb{N}_{0}}$, which represents the evolution of values in the system as the total sum of values of agents increases. We call this a \emph{competing growth process}. 

\begin{rmq}
The main reason for analysing the process $(v_{a}(\tau_{n}))_{a \in [A], n \in \mathbb{N}_{0}}$, and not the continuous time process $(v_a(t))_{a \in [A], t \in [0, \infty)}$, is that we may have $\lim_{n \to \infty} \tau_{n} < \infty$, i.e., explosion may occur. In many applications, there may be a total cap on the total possible sum of  `values'. For example, if the agents represent companies, there may be a cap on the total possible wealth that may exist. Thus, more relevant is the behaviour of the model before `resources' run out -- i.e., the behaviour before explosion. 
\end{rmq}

\begin{rmq}
   As far as we are aware, the general process $(v_{a}(\tau_{n}))_{a \in [A], n \in \mathbb{N}_{0}}$ has not been studied in the literature before. However, this process generalises the balls-in-bin process in a similar manner to the way trees associated with \emph{Crump-Mode-Jagers} branching processes generalise `preferential attachment trees' - see, for example,~\cite{holmgren-janson}. There is a large amount of literature concerning Crump-Mode-Jagers branching processes.
\end{rmq}

\begin{defn} \label{def:leadership-class}
    With regards to the process $(v_{a}(\tau_{n}))_{a \in [A], n \in \mathbb{N}_{0}}$:
    \begin{itemize}
        \item We say \emph{leadership} occurs if for some $a \in [A]$, we have $\max_{a' \in [A]} v_{a'}(\tau_{n}) = v_{a}(\tau_{n})$ for all but finitely many $n$.
        \item We say \emph{strict leadership} occurs, if for some \emph{unique} $a \in A$, $\max_{a' \in [A]} v_{a'}(\tau_{n}) = v_{a}(\tau_{n})$ for all but finitely many $n$. 
        \item Finally, we say \emph{monopoly} occurs if there is a unique $a \in A$ such that $\lim_{n \to \infty} v_{a}(\tau_{n}) = \infty$.
    \end{itemize}
\end{defn}
\noindent Clearly the occurence of monopoly implies strict leadership, which in turn implies leadership. 

The results in this paper involve criteria related to convergence of random series. Recall that, given a sequence of mutually independent random variables $(S_{j})_{j \in \mathbb{N}}$, by the Kolmogorov $0$-$1$ law, the series $\sum_{j=1}^{\infty}S_{j}$ either converges almost surely or diverges almost surely. Necessary and sufficient conditions for this convergence are provided by the classical Kolmogorov three series theorem - see Section~\ref{sec:aux}. Note also, that divergence of a series $\sum_{j=1}^{\infty} X_{j}$ of non-negative terms is equivalent to $\sum_{j=1}^{\infty} X_{j} = \infty$.

The first main result is a general classification of leadership in the process $(v_{a}(\tau_{n}))_{a \in [A], n \in \mathbb{N}_{0}}$. 
Suppose that $(X_{j})_{j \in \mathbb{N}}, (X'_{j})_{j \in \mathbb{N}}$ denote generic, independent sequences with $(X_{j})_{j \in \mathbb{N}} \sim (X'_{j})_{j \in \mathbb{N}} \sim (X^{\ssup{1}}_{j})_{j \in \mathbb{N}}$.

\begin{thm} \label{thm:leadership-classification}
    In the process $(v_{a}(\tau_{n}))_{a \in [A], n \in \mathbb{N}_{0}}$, we have the following:
    \begin{enumerate}
        \item \label{item:leadership} Leadership occurs with probability zero or one. Moreover, leadership occurs with probability one if and only if the random series $\sum_{j=1}^{\infty} (X_{j} - X'_{j})$ converges almost surely.
        \item \label{item:monopoly} If $\sum_{j=1}^{\infty} X_{j}$ diverges almost surely, then monopoly occurs with probability zero. On the other hand, if $\sum_{j=1}^{\infty} X_{j}$ converges almost surely, and either
        \begin{enumerate}
            \item \label{item:no-atom} For any $a \neq a' \in [A]$, $\exists j > v_{a}(0) \wedge v_{a'}(0)$ such that $X_{j}$ has no atom on $[0, \infty)$, or,
            \item \label{item:no-det-sequence} For any sequence $(c_{j})_{j \in \mathbb{N}} \in [0, \infty)^\mathbb{N}$ , we have $\sum_{j=1}^{\infty} \Prob{X_{j} \neq c_{j}} = \infty$,
        \end{enumerate}
        then monopoly occurs with probability one.
        \item \label{item:strict-leadership} If the series $\sum_{j=1}^{\infty} (X_{j} - X'_{j})$ converges almost surely, and, for any $a \neq a' \in [A]$
        \begin{equation} \label{eq:borel-cantelli}
        \sum_{k=1}^{\infty} \Prob{0 \leq \sum_{j=v_{a'}(0)+1}^{k} (X'_{j} - X_{j}) - \sum_{j= v_{a}(0)+1}^{v_{a'}(0)} X_{j} < X_{k+1}} < \infty,
        \end{equation}
        strict leadership occurs with probability one.
    \end{enumerate}
\end{thm}

\begin{rmq} \label{rem:novelty}
    Convergence of random series have been used in proving sufficiency of $\sum_{j=0}^{\infty} 1/f(j)^2$ for leadership in the balls-in-bins model (see, for example~\cite{khanin-neuron-polarity, oliveira-brownian-motion}). The main technicality of this article lies in extending these criteria to the general setting, using Theorem~\ref{thm:no-atom} -- the proof is comparatively much simpler when we can assume that distribution of one of the $X_{j}$ has no atom on $[0, \infty)$. Meanwhile, in the proof of necessity, we use Proposition~\ref{prop:fluctuating-sum} -- the techniques used for this direction in the balls-in-bins model (or related preferential attachment models), involving normal approximation or approximation by Brownian motion~\cite{khanin-neuron-polarity, oliveira-brownian-motion, dereich-morters-persistence, banerjee-bhamidi}, do not extend to the general setting where variances of waiting times may not exist. 
\end{rmq}
\begin{rmq} \label{rem:non-zero-one}
    Item~\ref{item:leadership} of Theorem~\ref{thm:leadership-classification} shows that leadership occurs with probability $0$ or $1$ in general. This does not seem to be immediately clear from applying, for example, standard zero-one laws. Moreover, it is actually the case that monopoly and strict leadership are \emph{not} zero-one events in general, this is the content of Proposition~\ref{prop:counter}, stated and proved in Section~\ref{sec:general}. Intuitively, the added requirement of having a `unique' leader removes the zero-one property. In addition, note that leadership is a property that does not depend on the initial condition, whereas this is not the case for monopoly or strict leadership. 
\end{rmq}

\begin{rmq}
    In the definition of $(v_{a}(\tau_{n}))_{a \in [A], n \in \mathbb{N}_{0}}$ the sequences $(X^{\ssup{a}}_{j})_{j \in \mathbb{N}}$ are identically distributed across $a \in [A]$. We believe it may be possible to extend the results, to some extent, to the case where the sequences $(X^{\ssup{a}}_{j})_{j \in \mathbb{N}}$ are independent, but not identically distributed, random variables, but do not endeavour to do this here. Such processes would generalise balls-in-bins-models with asymmetric feedback (see, e.g.,~\cite{balls-in-bins-asymmetric-feedback}). 
\end{rmq}

The following corollary provides criteria for strict leadership which may be easier to verify: 

\begin{cor} \label{cor:strict-leadership}
    In the process $(v_{a}(\tau_{n}))_{a \in [A], n \in \mathbb{N}_0}$, if the series $\sum_{j=1}^{\infty} (X_{j} - X'_{j})$ converges almost surely, $\sum_{j=1}^{\infty} \Prob{X_{j} - X'_{j} \neq 0} = \infty$, 
    and, for any $\eps > 0$ 
    \begin{equation} \label{eq:borel-cantelli-any-eps}
        \sum_{j=1}^{\infty} \Prob{X_{j} > \eps} < \infty, 
    \end{equation}
    then strict leadership occurs with probability one.
\end{cor}

\subsection{Applications to balls-in-bins models} \label{sec:balls-in-bins}

As previously mentioned, a particular instance of the model $(v_{a}(\tau_{n}))_{a \in [A], n \in \mathbb{N}_{0}}$ is the \emph{balls-in-bins} model with, possibly random feedback function (see also Proposition~\ref{prop:first-equivalence} for explicit proof). Suppose one begins at time zero with $A$ urns, $(u_{a}(0))_{a \in [A]}$, with each $u_{a}(0) \in \mathbb{N}$, representing the number of \emph{balls} that urn contains. To each urn $a\in [A]$ is a mutually independent sequence of random variables $(F_{a}(j))_{j \in \mathbb{N}_{0}}$, where, for each $a, a' \in [A]$, we have $(F_{a}(j)_{j \in \mathbb{N}_{0}} \sim (F_{a'}(j))_{j \in \mathbb{N}_{0}}$.  

\begin{defn} \label{def:urn-defn}
Given an initial condition $(u_{a}(0))_{a \in [A]}$, and random variables $(F_{a}(j))_{a \in [A], j \in \mathbb{N}_{0}}$, the balls-in-bins process with feedback $(u_{a}(n))_{a \in [A], n \in \mathbb{N}_{0}}$ is the discrete time Markov process defined recursively as follows. 
\begin{enumerate}
    \item At time $0$, sample the random variables $(F_{a}(u_{a}(0)))_{a \in [A]}$.
    \item At time $n+1$, given the collection $(u_{a}(n))_{a \in [A]}$, and $(F_{a}(u_{a}(n)))_{a \in [A]}$, sample an urn $a$ with probability 
    \[
    \frac{F_{a}(u_{a}(n))}{\sum_{a=1}^{A} F_{a}(u_{a}(n))}.
    \]
    Set $u_{a}(n+1) = u_{a}(n) + 1$, and for $a' \neq a$, set $u_{a'}(n+1) = u_{a'}(n)$. Sample $F_{a}(u_{a}(n+1))$. 
\end{enumerate}
\end{defn}

Suppose $(F(j))_{j \in \mathbb{N}_{0}}$, $(F'(j))_{j \in \mathbb{N}_{0}}$ denote generic, independent sequences of independent random variables with $(F(j))_{j \in \mathbb{N}_{0}} \sim (F'(j))_{j \in \mathbb{N}_{0}} \sim (F_{a}(j))_{j \in \mathbb{N}_0}$, for $a \in [A]$.
Suppose that the events of leadership, strict leadership and monopoly from Definition~\ref{def:leadership-class} are defined analogously for the balls-in-bins model with feedback $(u_{a}(n))_{a \in [A], n \in \mathbb{N}_{0}}$.

\begin{thm} \label{thm:balls-in-bins}
    In the balls-in-bins model with feedback $(u_{a}(n))_{a \in [A], n \in \mathbb{N}_{0}}$, with feedback function given by~$(F(j))_{j \in \mathbb{N}_{0}}$, we have the following statements: 
    \begin{enumerate}
        \item \label{item:balls-monopoly} Monopoly occurs with probability zero or one. Moreover, monopoly occurs with probability one if and only if $\sum_{j=0}^{\infty}1/F(j) < \infty$ almost surely.
        \item \label{item:balls-leadership} Strict leadership occurs with probability zero or one. Moreover, strict leadership occurs with probability one if and only if $\sum_{j=0}^{\infty} 1/F(j)^2 < \infty$ almost surely.  
    \end{enumerate}    
\end{thm}

\begin{rmq}
    Whilst Item~\ref{item:balls-monopoly} of Theorem~\ref{thm:balls-in-bins} has not been stated in the context with random feedback functions before, the proof of Item~\ref{item:balls-monopoly} uses essentially the same argument as that from Rubin~\cite{davis-rubin}. However, we include its proof for completeness. Moreover, as mentioned in Remark~\ref{rem:novelty} one of the directions of Item~\ref{item:balls-leadership} uses similar ideas to~\cite{khanin-neuron-polarity, oliveira-brownian-motion}, but the general result requires novel techniques. 
\end{rmq}

\subsection{Auxiliary results on the convergence of random series} \label{sec:aux}

As the results of this paper rely heavily on criteria for convergence of random series of independent random variables, we recall here the well-known Kolmogorov three series theorem:
\begin{thm*}[Kolmogorov three series theorem, e.g.{~\cite[Theorem~2.5.8., page~73]{durrett}}] 
 For a sequence of mutually independent random variables $(S_{j})_{j \in \mathbb{N}}$, let $C > 0$ be given. Then the series $\sum_{j=1}^{\infty} S_{j}$ converges almost surely if and only if
 \begin{equation} \label{eq:kolmogorov-three-series}
     \sum_{j=1}^{\infty} \Prob{\left|S_{j}\right| > C} < \infty, \quad \sum_{j=1}^{\infty} \E{S_{j} \mathbf{1}_{\left|S_{j}\right| \leq C}} < \infty \quad \text{and} \quad \sum_{j=1}^{\infty} \var{S_{j} \mathbf{1}_{\left|S_{j}\right| \leq C}} < \infty
 \end{equation}
\end{thm*}
\begin{rmq} \label{rmq:two-series}
    Recall that, if $\sum_{j=1}^{\infty} \E{S_{j}} < \infty$ and $\sum_{j=1}^{\infty} \var{S_{j}} < \infty$ then $\sum_{j=1}^{\infty} S_{j}$ converges almost surely - this is the Kolmogorov two series theorem. 
    Note also that, if each $S_{j} \geq 0$ almost surely, then one only needs to check convergence of the first two series in~\eqref{eq:kolmogorov-three-series} - see Proposition~\ref{prop:three-series-positive}. 
\end{rmq}

The proofs of Theorem~\ref{thm:leadership-classification} and Corollary~\ref{cor:strict-leadership} use Theorem~\ref{thm:no-atom} whilst the proof of Theorem~\ref{thm:leadership-classification} also uses Proposition~\ref{prop:fluctuating-sum}. We believe that Theorem~\ref{thm:no-atom} may be of interest in its own right. 

\begin{thm} \label{thm:no-atom}
    Suppose that $(Y_{j})_{j \in \mathbb{N}}$ is a sequence of mutually independent random variables such that $\sum_{j=1}^{\infty} Y_{j}$ converges almost surely. Then the distribution of $\sum_{j=1}^{\infty} Y_{j}$ contains an atom on $\mathbb{R}$ if and only if, for some collection $(c_{j})_{j \in \mathbb{N}} \in \mathbb{R}^{\mathbb{N}}$
        \begin{equation} \label{eq:atomic-sum-const}
            \forall j \in \mathbb{N} \quad \Prob{Y_{j} = c_{j}} > 0 \quad \text{ and } \quad \sum_{j=1}^{\infty} \Prob{Y_{j} \neq c_{j}} < \infty. 
        \end{equation}
\end{thm}

\begin{rmq}
    Equation~\eqref{eq:atomic-sum-const} in Theorem~\ref{thm:no-atom}, together with the Borel--Cantelli lemma shows that a convergent series of random variables $\sum_{j=1}^{\infty} Y_{j}$ contains an atom, if and only if, the random variables $(Y_{j})_{j \in \mathbb{N}}$ coincide with deterministic values $(c_{j})_{j \in \mathbb{N}}$ for all but finitely many $j$.
\end{rmq}

\begin{rmq} \label{rem:symm-constants}
Note that, if the random variables $(Y_{j})_{j \in \mathbb{N}}$ are symmetric, Equation~\eqref{eq:atomic-sum} can only be satisfied if 
\begin{equation} \label{eq:all-but-finitely-many-0}
\sum_{j=1}^{\infty} \Prob{Y_{j} \neq 0} < \infty. 
\end{equation}
Indeed, Equation~\eqref{eq:atomic-sum} implies that, for $j$ sufficiently large, $\Prob{Y_{j} = c_j} = \Prob{Y_{j} = - c_{j}} > 1/2$, which is impossible if infinitely many values of $c_{j}$ are non-zero. 
\end{rmq}

Finally, we have the following useful proposition:

\begin{prop}[Divergence of series of symmetric random variables] \label{prop:fluctuating-sum}
    Suppose that $(S_{j})_{j \in \mathbb{N}}$ is a sequence of mutually independent, symmetric random variables. If $\sum_{j=1}^{\infty} S_{j}$ diverges almost surely, then 
        \begin{equation} \label{eq:varying-sum}
        \limsup_{n \to \infty} \sum_{j=1}^{n} S_j = \infty \quad \text{and} \quad \liminf_{n \to \infty} \sum_{j=1}^{n} S_{j} = -\infty \quad \text{almost surely.}
        \end{equation}
\end{prop}

\section{Proofs of results} \label{sec:proofs}
\subsection{Notation}
In what follows, it will be helpful to have explicit notation for the events of leadership, strict leadership and monopoly. 
For an agent $a \in [A]$, we define the events of \emph{leadership} of $a$, $\lead_{a}$, \emph{monopoly} of $a$, $\mon_{a}$,  and \emph{strict leadership} of $a$, $\slead_a$ respectively by 
\begin{equation} \label{eq:leadership}
\lead_{a} := \left\{\exists N \in \mathbb{N}\colon \; \forall n \geq N \; \forall a' \in [A] \; a' \neq a \implies v_{a}(\tau_{n}) \geq v_{a'}(\tau_n) \right\}. 
\end{equation}
\begin{equation} \label{eq:monopoly-def}
    \mon_{a} := \left\{\exists N \in \mathbb{N}\colon \; \forall n \geq N \; v_{a}(\tau_{n}) = v_{a}(\tau_{N}) \right\}, 
\end{equation}
and
\begin{equation} \label{eq:strong-leadership}
\slead_{a} := \left\{\exists N \in \mathbb{N}\colon \; \forall n \geq N \; \forall a' \in [A] \; a' \neq a \implies v_{a}(\tau_{n}) > v_{a'}(\tau_{n}) \right\}. 
\end{equation} 
Thus, leadership coincides with the event $\bigcup_{a \in [A]} \lead_{a}$, monopoly with $\bigcup_{a \in [A]} \mon_{a}$ and strict leadership with $\bigcup_{a \in [A]} \slead_{a}$.

We assume throughout that `empty' sums (for example $\sum_{j=1}^{0} X_{j}$) are zero. Moreover, given a collection of events $(\mathcal{E}_{j})_{j \in \mathbb{N}}$ in a probability space, we use the abbreviation ``$\mathcal{E}_{j} \text{ i. o.}$'' to denote the event that infinitely many $\mathcal{E}_{j}$ occur, i.e., $\bigcap_{n=1}^{\infty} \bigcup_{j=n}^{\infty} \mathcal{E}_{j}$ occurs. 

\subsection{Proofs of Theorem~\ref{thm:leadership-classification}, Corollary~\ref{cor:strict-leadership} and Proposition~\ref{prop:counter}} \label{sec:general}
We start with the proof of Theorem~\ref{thm:leadership-classification}, which is self-contained, except for usage of Theorem~\ref{thm:no-atom} and Proposition~\ref{prop:fluctuating-sum}. The proofs of Corollary~\ref{cor:strict-leadership} and Proposition~\ref{prop:counter} use ideas, and sometimes notation, from the proof of Theorem~\ref{thm:leadership-classification}. 

\begin{proof}[Proof of Item~\ref{item:leadership} of Theorem~\ref{thm:leadership-classification}]
For Item~\ref{item:leadership}, suppose, first, that $\sum_{j=1}^{\infty} (X_{j} - X'_{j})$ converges almost surely. We show that given any two agents, eventually one of them has a value that does not exceed the other, i.e. one is a `winner'. Thus, by iteratively ordering agents, we can find a leader. 

In this regard, suppose $a, a' \in [A]$ are agents. 
Then we define 
\[
\win(a,a') := \left\{\exists N \in \mathbb{N}\colon \; \forall n \geq N \;  v_{a}(\tau_{n}) \geq v_{a'}(\tau_{n}) \right\}.
\]
Note that $\win(\cdot, \cdot)$ leads to a transitive relation, in the sense that, for $a, a', a'' \in [A]$ 
\[
\win(a,a') \cap \win(a', a'') \subseteq \win(a,a''). 
\]
Our goal is to show that the relation is complete, that is, 
\begin{equation} \label{eq:complete-relation}
    \Prob{\bigcap_{a \neq a'} \left(\win(a,a') \cup \win(a', a) \right)} = 1, 
\end{equation}
so that, by iteratively ordering elements, the event $\bigcup_{a \in [A]} \bigcap_{a' \neq a} \win(a, a') = \bigcup_{a \in [A]} \lead_{a}$ occurs. 
For~\eqref{eq:complete-relation}, suppose that $a, a' \in [A]$ with $a \neq a'$, and assume, without loss of generality, that $v_{a}(0) \leq v_{a'}(0)$. Now, if $\lim_{n \to \infty} v_{a}(\tau_{n}) < \infty$ or $\lim_{n \to \infty} v_{a'}(\tau_{n}) < \infty$, note that either $\win(a,a')$ occurs or $\win(a',a)$ occurs, depending on whether or not $\lim_{n \to \infty} v_{a}(\tau_{n}) \geq \lim_{n \to \infty} v_{a'}(\tau_{n})$. Thus, we may assume that $\lim_{n \to \infty} v_{a}(\tau_{n}) = \infty$ and $\lim_{n \to \infty} v_{a'}(\tau_{n}) = \infty$. There are now two cases. 
First, suppose that $\sum_{j=1}^{\infty}\Prob{X^{\ssup{a}}_{j} \neq X^{\ssup{a'}}_{j}} = \infty$. Then, noting that the random variables $X_{j} - X'_{j}$ are symmetric, by Theorem~\ref{thm:no-atom} and Remark~\ref{rem:symm-constants}, the almost surely convergent sum 
\begin{equation} \label{eq:random-conv-sum}
\sum_{j=v_{a}(0)+1}^{v_{a'}(0)} X^{\ssup{a}}_{j} + \sum_{j= v_{a'}(0)+1}^{\infty} \left(X^{\ssup{a}}_{j} - X^{\ssup{a'}}_{j} \right), 
\end{equation}
contains no atom at $0$, and thus is almost surely strictly positive or negative. Therefore, with probability one, for all $k$ sufficiently large, 
\begin{equation} \label{eq:pos-neg-sum}
\text{either} \quad 
\sum_{j=v_{a}(0)+1}^{k} X^{\ssup{a}}_{j}  < \sum_{j=v_{a'}(0)+1}^{k} X^{\ssup{a'}}_{j} \quad \text{or} \quad \sum_{j=v_{a'}(0)+1}^{k} X^{\ssup{a'}}_{j}  < \sum_{j=v_{a}(0)+1}^{k} X^{\ssup{a}}_{j}. 
\end{equation}
The partial sum $\sum_{j=v_{a}(0)+1}^{k} X^{\ssup{a}}_{j}$ represents the time taken for the the value of $a$ to reach $k$, and similarly for $a'$. Thus~\eqref{eq:pos-neg-sum} implies that one of $a$ or $a'$ is a winner, i.e., $\Prob{\win(a,a') \cup \win(a', a)} = 1$. As the finite intersection of almost sure events,~\eqref{eq:complete-relation} follows.
If, however, $\sum_{j=1}^{\infty}\Prob{X^{\ssup{a}}_{j} \neq X^{\ssup{a'}}_{j}} < \infty$, then by the Borel-Cantelli lemma, $X^{\ssup{a}}_{j} = X^{\ssup{a'}}_{j}$ for all but finitely many $j$. This implies that eventually, the partial sums converging to~\eqref{eq:random-conv-sum} are constant: the convergent series has all but finitely many terms equal to zero. Thus, for all $k$ sufficiently large, either~\eqref{eq:pos-neg-sum} holds, or 
\begin{equation} \label{eq:ident-sum}
\sum_{j=v_{a}(0)+1}^{k} X^{\ssup{a}}_{j}  = \sum_{j=v_{a'}(0)+1}^{k} X^{\ssup{a'}}_{j}. 
\end{equation}
As $\win(a,a') \cup \win(a', a)$ also allows for the case that \emph{both} $a$ and $a'$ have the same values for $n$ sufficiently large, again, this implies $\Prob{\win(a,a') \cup \win(a', a)} = 1$, and thus implies~\eqref{eq:complete-relation}. 

For the converse direction of Item~\ref{item:leadership}, suppose that $\sum_{j=1}^{\infty} (X_{j} - X'_{j})$ diverges almost surely, and again, assume without loss of generality that $v_{a}(0) \leq v_{a'}(0)$. Then by Proposition~\ref{prop:fluctuating-sum}, with probability one we have 
\begin{linenomath}
\begin{align}
& \limsup_{k \to \infty} \sum_{j=v_{a}(0)+1}^{v_{a'}(0) } X^{\ssup{a}}_{j} + \sum_{j= v_{a'}(0) + 1}^{k} \left(X^{\ssup{a}}_{j} - X^{\ssup{a'}}_{j} \right) = \infty \quad \text{and} \\ & \hspace{3cm} \liminf_{k \to \infty} \sum_{j=v_{a}(0)+1}^{v_{a'}(0)} X^{\ssup{a}}_{j} + \sum_{j= v_{a'}(0)+1}^{k} \left(X^{\ssup{a}}_{j} - X^{\ssup{a'}}_{j} \right) = -\infty. 
\end{align}
\end{linenomath}
Therefore, we have both 
\begin{linenomath}
\begin{align}
&    \sum_{j=v_{a}(0)+1}^{k} X^{\ssup{a}}_{j}  < \sum_{j=v_{a'}(0)+1}^{k} X^{\ssup{a'}}_{j} \quad \text{for infinitely many $k$, and }
 \\ & \hspace{3cm}   \sum_{j=v_{a}(0)+1}^{k} X^{\ssup{a}}_{j}  > \sum_{j=v_{a'}(0)+1}^{k} X^{\ssup{a'}}_{j} \quad \text{for infinitely many $k$.}
\end{align}
\end{linenomath}
This implies that $\Prob{\win(a,a') \cup \win(a', a)} = 0$, and, since $a, a'$ are arbitrary, that
\[
\Prob{\bigcup_{a \in [A]} \lead_{a}} = \Prob{\bigcup_{a \in [A]} \bigcap_{a' \neq a} \win(a, a')} = 0.
\]
\end{proof}
\begin{proof}[Proof of Item~\ref{item:monopoly} of Theorem~\ref{thm:leadership-classification}]
Suppose that, for each $a \in [A]$ we define 
\begin{equation} \label{eq:sigma-def}
\sigma(a) := \sum_{j=v_{a}(0)+1}^{\infty} X^{\ssup{a}}_{j}, \quad \text{i.e., the `first time' that $a$ attains infinite value.}    
\end{equation}
 Since $[A]$ is finite, there must be some agent of infinite value as $\sum_{a=1}^{A} v_a(\tau_{n}) \to \infty$, hence, by definition of the times $(\tau_{n})_{n \in \mathbb{N}}$. 
\begin{equation} \label{eq:explosion-attained}
\tau_{\infty} = \lim_{n \to \infty} \tau_{n} = \min_{a \in [A]} \sigma(a). 
\end{equation}
Now, for the first statement of Item~\ref{item:monopoly}, suppose that $\sum_{j=1}^{\infty} X_{j} = \infty$ almost surely. Then \[\Prob{\bigcap_{a \in [A]} \left\{ \sigma(a) = \infty \right\}} = 1,\] so that $\tau_{\infty} = \infty$ almost surely. Since, for each $a \in [A]$, the values $X^{\ssup{a}}_{j}$ are finite, as the finite sum of finite random variables, we have 
\begin{equation}
\sum_{j=v_{a}(0)+1}^{k} X^{\ssup{a}}_{j} < \tau_{\infty} = \infty \quad \text{for each $k \in \mathbb{N}$}. 
\end{equation}
Therefore, for each $a \in [A]$, $\lim_{n \to \infty} v_{a}(\tau_{n}) = \infty$. As monopoly entails that there is only a single agent with value tending to infinity, this implies that $\Prob{\bigcup_{a \in [A]} \mon_{a}} = 0$. 

For the converse statement, suppose that $\sum_{j=1}^{\infty} X_{j} <  \infty$ almost surely, and one of the statements Item~\ref{item:no-atom} or Item~\ref{item:no-det-sequence} from the theorem are satisfied. Then, for any $a, a' \in [A]$ with $a \neq a'$, (by applying Theorem~\ref{thm:no-atom} if Item~\ref{item:no-det-sequence} is satisfied), the random variable 
\[
\sum_{j=v_{a}(0)+1}^{\infty} X^{\ssup{a}}_{j} - \sum_{j=v_{a'}(0)+1}^{\infty} X^{\ssup{a'}}_{j}
\]
contains no atom on $\mathbb{R}$, hence $\Prob{\sigma(a) = \sigma(a')} = 0$.  This implies $\Prob{\bigcap_{a,a' \in [A], a \neq a'} \left\{\sigma(a) = \sigma(a')\right\}} = 0$, hence, almost surely, there exists a unique $a^{*} \in [A]$ such that,  $\sigma(a^{*}) = \tau_{\infty}$. By uniqueness, for any $a' \neq a^*$, $v_{a'}(\tau_{\infty}) < \infty$, thus, $\lim_{n \to \infty} v_{a'}(\tau_{n}) = c_{a'}$, for some $c_{a'} \in \mathbb{N}$. This implies that $\mon_{a^{*}}$ occurs. 
\end{proof}
\begin{proof}[Proof of Item~\ref{item:strict-leadership} of Theorem~\ref{thm:leadership-classification}]
For Item~\ref{item:strict-leadership} of Theorem~\ref{thm:leadership-classification}, note first that by Item~\ref{item:leadership} of Theorem~\ref{thm:leadership-classification}, $\Prob{\bigcup_{a \in [A]} \lead_{a}} = 1$. Any candidate for strict leadership must also be an eventual leader of the process.  Suppose then that for some $a \in [A]$, $\lead_{a}$ occurs, but $\slead_{a}$ does not. Then there must be at least one $a' \in [A]$ such that $v_{a}(\tau_{n}) = v_{a'}(\tau_{n})$ for infinitely many $n \in \mathbb{N}$, and since $a$ is a leader, $\lim_{n \to \infty} v_{a}(\tau_{n}) = \infty$. Thus, there must be infinitely many $k$ such that 
\begin{equation} \label{eq:inf-many-k}
\sum_{j=v_{a}(0) + 1}^{k} X^{\ssup{a}}_{j} \leq \sum_{j=v_{a'}(0) + 1}^{k} X^{\ssup{a'}}_{j} < \sum_{j=v_{a}(0) + 1}^{k+1} X^{\ssup{a}}_{j}. 
\end{equation}
Therefore, 
\begin{linenomath}
\begin{align} \label{eq:constant-catch-up}
\nonumber 
& \Prob{\left(\bigcup_{a \in [A]} \lead_{a}\right) \setminus \left(\bigcup_{a \in [A]} \slead_{a}\right)} \\ \nonumber & \hspace{1cm}  \leq \Prob{\bigcup_{a\neq a'} \left\{\sum_{j=v_{a}(0) + 1}^{k} X^{\ssup{a}}_{j} \leq \sum_{j=v_{a'}(0) + 1}^{k} X^{\ssup{a'}}_{j} < \sum_{j=v_{a}(0) + 1}^{k+1} X^{\ssup{a}}_{j} \quad  \text{ i.o. } \right\}} \\ & \hspace{1cm} \leq \sum_{a \neq a'} \Prob{0 \leq \sum_{j=v_{a'}(0) + 1}^{k} X^{\ssup{a'}}_{j} - \sum_{j=v_{a}(0) + 1}^{k} X^{\ssup{a}}_{j} < X^{\ssup{a}}_{k+1} \quad \text{ i.o. }}. 
\end{align}
\end{linenomath}
But then by Equation~\eqref{eq:borel-cantelli} and the Borel-Cantelli lemma, each summand on the right-side of~\eqref{eq:constant-catch-up} is $0$, hence proving the claim. 
\end{proof}

The proof of Corollary~\ref{cor:strict-leadership} follows similar logic to the proof of Item~\ref{item:strict-leadership} of Theorem~\ref{thm:leadership-classification}. 

\begin{proof}[Proof of Corollary~\ref{cor:strict-leadership}]
    Suppose that, with positive probability, for some $a \in [A]$, $\lead_{a}$ occurs but $\slead_{a}$ does not. Then by~\eqref{eq:inf-many-k}, for some $a' \neq a$ there exists infinitely many $k$ such that
    \begin{equation} \label{eq:sum-bound}
        \sum_{j=v_{a'}(0) + 1}^{k} X^{\ssup{a'}}_{j} - \sum_{j=v_{a}(0) + 1}^{k} X^{\ssup{a}}_{j} < X^{\ssup{a}}_{k+1}. 
    \end{equation}
    Now, by Equation~\eqref{eq:borel-cantelli-any-eps}, and the Borel-Cantelli lemma, for any $\eps > 0$ with probability $1$, $X^{\ssup{a}}_{k+1} \leq \eps$ for all but finitely many $k$. Assume, first that $v_{a'}(0) \geq  v_{a}(0)$. Then by~\eqref{eq:sum-bound}, taking limits as $k \to \infty$, the almost surely convergent sum 
    \[
     \sum_{j=v_{a'}(0) + 1}^{\infty} \left(X^{\ssup{a'}}_{j} - X^{\ssup{a}}_{j} \right) + \sum_{j=v_{a}(0) + 1}^{v_{a'}(0)} \left(X^{\ssup{a'}}_{j} - X^{\ssup{a}}_{j} \right)
    \]
    must equal zero, almost surely. But now, as $\sum_{j=1}^{\infty} \Prob{X_{j} - X'_{j} \neq 0} = \infty$, Theorem~\ref{thm:no-atom} and Remark~\ref{rem:symm-constants} imply that the series above has no atom on $\mathbb{R}$, this leads to a contradiction. The case $v_{a'}(0) <  v_{a}(0)$ is similar. 
\end{proof}

Finally, we prove Proposition~\ref{prop:counter}, which shows that monopoly and strict leadership are not, in general, zero-one events. 
\begin{prop} \label{prop:counter}
    We have the following:
    \begin{enumerate}
        \item \label{item:mon-counter} There exists a competing birth process $(v_{a}(\tau_{n}))_{a \in [A], n \in \mathbb{N}_0}$ such that $0 < \Prob{\bigcup_{a \in [A]} \mon_{a}} < 1.$
        \item \label{item:slead-counter} There exists a competing birth process $(v_{a}(\tau_{n}))_{a \in [A], n \in \mathbb{N}_0}$ such that $0 < \Prob{\bigcup_{a \in [A]} \slead_{a}} < 1.$
    \end{enumerate}
\end{prop}
\begin{proof}[Proof of Proposition~\ref{prop:counter}]
For both examples we assume, for simplicity, that $N = 2$ and $v_{1}(0) = v_{2}(0) = 0$. For Item~\ref{item:mon-counter} of Proposition~\ref{prop:counter}, we choose the values $(X_{j})_{j \in \mathbb{N}}$ such that each $X_{j}$ is concentrated on two values, with 
\[
    X_{j} = \begin{cases}
        2^{-j} & \text{ with probability $1- \frac{1}{(j+1)^2}$} \\
        0 & \text{otherwise}.
    \end{cases}
\]
Then with $\sigma(1), \sigma(2)$ as defined in~\eqref{eq:sigma-def}, note that $\sigma(a) \leq 1$ almost surely, for $a \in \{1, 2\}$. Following similar logic to the proof of Item~\ref{item:monopoly} of Theorem~\ref{thm:leadership-classification}, we have 
\begin{linenomath}
\begin{align*}
\Prob{\mon_{1}} & \geq \Prob{\sigma(1) = \frac{1}{2}, \sigma(2) = 1} \\ & \geq \Prob{\left\{X^{\ssup{1}}_{1} = 0 \right\}\cap \left(\bigcap_{j=2}^{\infty} \left\{X^{\ssup{1}}_{j} = 2^{-j}\right\} \right)} \Prob{\bigcap_{j=1}^{\infty} \left\{ X^{\ssup{2}}_{j} = 2^{-j}\right\}}
\\ & = \left(\frac{1}{4} \prod_{j=2}^{\infty} \left(1 - \frac{1}{(j+1)^2}\right)\right) \prod_{j=1}^{\infty} \left(1 - \frac{1}{(j+1)^2}\right) > 0
\end{align*}
\end{linenomath}
where we have used the independence of the values $(X^{\ssup{a}}_{j})_{a \in \{1,2\}, j \in \mathbb{N}}$. The last inequality follows from the fact that the infinite products consist of non-zero terms and $\sum_{j=1}^{\infty} \frac{1}{(j+1)^2} < \infty$. 
On the other hand, by similar computations, we also have 
\begin{linenomath}
\begin{align*}
\Prob{\left(\bigcup_{a=1}^{2} \mon_{a}\right)^c} = \Prob{\sigma(1) = \sigma(2)} & \geq \Prob{\sigma(1) = 1, \sigma(2) = 1} \\ & \geq \Prob{\bigcap_{j=2}^{\infty} \left\{X^{\ssup{1}}_{j} = 2^{-j}\right\}} \Prob{\bigcap_{j=1}^{\infty} \left\{ X^{\ssup{2}}_{j} = 2^{-j}\right\}}
\\ & = \left(\prod_{j=1}^{\infty} \left(1 - \frac{1}{(j+1)^2}\right)\right)^2 > 0.
\end{align*}
\end{linenomath}
Item~\ref{item:mon-counter} of Proposition~\ref{prop:counter} follows. 

For Item~\ref{item:slead-counter} of Proposition~\ref{prop:counter}, suppose that each $X_{j}$ takes values in $\mathbb{N}$.  More concretely, suppose $X_{j} \sim \Geom{p_{j}}$ for $p_{j} \in (0,1)$, where $\Geom{p_{j}}$ denotes the geometric distribution with parameter $p_{j}$. For leadership, by Item~\ref{item:leadership} of Theorem~\ref{thm:leadership-classification}, we need $\sum_{j=1}^{\infty}(X_{j} - X'_{j}) < \infty$. Since, in this case,  $\Prob{|X_{j} - X'_{j}| \geq 1} = \Prob{X_{j} \neq X'_{j}}$, by the Kolmogorov three series theorem, it must be the case that 
\begin{equation} \label{eq:another-conv}
\sum_{j=1}^{\infty} \Prob{X_{j} \neq X'_{j}} < \infty. 
\end{equation}
Suppose we choose $p_{j}$ to satisfy~\eqref{eq:another-conv}, but also with each $\Prob{X^{\ssup{1}}_{j} = X^{\ssup{2}}_{j}} > 0$ (e.g. $p_{j} := (1 - \frac{1}{(j+1)^2})$). Then this implies that
\begin{equation} \label{eq:slead-1}
\Prob{\left(\bigcup_{a=1}^{2} \slead_{a}\right)^c} \geq \Prob{\forall j \in \mathbb{N} \; X^{\ssup{1}}_{j} = X^{\ssup{2}}_{j}} =  \prod_{j=1}^{\infty} \left(1 - \Prob{X^{\ssup{1}}_{j} \neq X^{\ssup{2}}_{j}} \right) > 0, 
\end{equation}
since $\sum_{j=1}^{\infty} \Prob{X^{\ssup{1}}_{j} \neq X^{\ssup{2}}_{j}} < \infty$ and each $\Prob{X^{\ssup{1}}_{j} \neq X^{\ssup{2}}_{j}} > 0$. On the other hand, 
  \begin{linenomath}
  \begin{align*}
   \Prob{\slead_{1}} & \geq \Prob{X^{\ssup{1}}_{1} < X^{\ssup{2}}_{1}} \Prob{\forall j \geq 2 \;  X^{\ssup{1}}_{j} = X^{\ssup{2}}_{j}} \\ & = \Prob{X^{\ssup{1}}_{1} < X^{\ssup{2}}_{1}} \prod_{j=2}^{\infty} \left(1 - \Prob{X_{j} \neq X'_{j}} \right) > 0, 
   \end{align*}
   \end{linenomath}
where the final inequalities use similar logic to~\eqref{eq:slead-1}. Item~\ref{item:slead-counter} of Proposition~\ref{prop:counter} follows. 
\end{proof}

\begin{rmq}
Suppose that all of the $X_{j}$ are concentrated on a set where the distance between any two points is bounded away from $0$, and, for all $a, a' \in [A]$ $v_{a}(0) = v_{a'}(0)$. By using similar ideas to the proof of Item~\ref{item:slead-counter} of Proposition~\ref{prop:counter} one can show that, if $\Prob{\bigcup_{a \in [A]} \lead_{a}} = 1$, then $0 < \Prob{\bigcup_{a \in [A]} \slead_{a}} < 1$. 
\end{rmq}

\subsection{Proof of Theorem~\ref{thm:balls-in-bins}}
We first relate the balls-in-bins process to a particular case of the process $(v_{a}(\tau_{n}))_{a \in [A], n \in \mathbb{N}_{0}}$. Suppose that $(F(j))_{j \in \mathbb{N}_{0}}$ is the feedback function associated with the balls-in-bins process. 
 Suppose $(X_{j})_{j \in \mathbb{N}}$ denotes a sequence of random variables with mixed distribution
\begin{equation}\label{eq:mixed-exp}
    X_{j} \sim \Exp{F(j-1)}. 
\end{equation}

\begin{prop} \label{prop:cont-equivalence}
    Suppose $(u_{a}(n))_{a \in [A], n \in \mathbb{N}_{0}}$ denotes a balls-in-bins process, with associated function $f$, and $(v_{a}(\tau_{n}))_{a \in [A], n \in \mathbb{N}_{0}}$ denotes a competing growth process, with random variables $(X_{i})_{i \in \mathbb{N}}$ distributed as~\eqref{eq:mixed-exp}. Suppose also that $(u_{a}(0))_{a \in [A]} = (v_{a}(0))_{a \in [A]}$, i.e., both have the same initial conditions. Then \[(u_{a}(n))_{a \in [A], n \in \mathbb{N}_{0}} \sim (v_{a}(\tau_{n}))_{a \in [A], n \in \mathbb{N}_{0}}.\] 
\end{prop}

The proof of Proposition~\ref{prop:cont-equivalence}, along with Item~\ref{item:balls-monopoly} of Theorem~\ref{thm:balls-in-bins} are straightforward extensions of their analogues where $(F(j))_{j \in \mathbb{N}_0}$ is deterministic (see, e.g., \cite{oliveira-thesis}). However, we include their proofs for completeness. 

We include the proof of Proposition~\ref{prop:cont-equivalence} for completeness.
\begin{proof}[Proof of Proposition~\ref{prop:cont-equivalence}]
    We construct the process $(v_{a}(\tau_{n}))_{a \in [A], n \in \mathbb{N}_{0}}$, and show it has the same distribution as $(u_{a}(n))_{a \in [A], n \in \mathbb{N}_0}$. Note that, by mutual independence of~$(F_{a}(j))_{a \in [A], j \in \mathbb{N}_{0}}$, it makes no difference to the law of the total sequence $(u_{a}(n))_{a \in [A], n \in \mathbb{N}_{0}}$ if these are sampled initially, or iteratively as in Definition~\ref{def:urn-defn}. Therefore, in constructing $(v_{a}(\tau_{n}))_{a \in [A], n \in \mathbb{N}_{0}}$, we start by sampling one layer of randomness: the collection $(F_{a}(j))_{a \in [A], j \in \mathbb{N}_{0}}$. Then conditional on these random variables, the values $(X^{\ssup{a}}_{j})_{a \in [A], j \in \mathbb{N}_{0}}$ are mutually independent, exponential random variables. Now, suppose that $(\mathcal{F}_{t})_{t \geq 0}$ denotes the filtration generated by the process $(v_a(t))_{a \in [A], t \in [0, \infty)}$, given $(F_{a}(j))_{a \in [A], j \in \mathbb{N}_{0}}$. Then:
    \begin{enumerate}
        \item By assumption, we have $(v_{a}(0))_{a \in [A]} = (u_{a}(0))_{a \in [A]}$. 
        \item Recalling the times $(\tau_{n})_{n \in \mathbb{N}}$, note that, for $a \in [N ]$, we have 
        \begin{linenomath}
        \begin{align} \label{eq:memoryless}
        & \nonumber \Prob{v_{a}(\tau_{n+1}) = v_{a}(\tau_{n}) + 1} = \Prob{X^{\ssup{a}}_{v_{a}(\tau_{n})+1} = \tau_{n+1} - \tau_{n}} = \E{\Prob{X^{\ssup{a}}_{v_{a}(\tau_{n})+1} = \tau_{n+1} - \tau_{n} \, \bigg | \, \mathcal{F}_{\tau_{n}}}}
        \\ & \hspace{0.5cm} = \E{\Prob{X^{\ssup{a}}_{v_{a}(\tau_{n})+1} = \min_{a \in [A]} \left\{X^{\ssup{a}}_{v_{a}(\tau_{n})+1} \right\} \, \bigg | \, \mathcal{F}_{\tau_{n}}}} = \Prob{X^{\ssup{a}}_{v_{a}(\tau_{n})+1} = \min_{a \in [A]} \left\{X^{\ssup{a}}_{v_{a}(\tau_{n})+1} \right\}}.
        \end{align}
        \end{linenomath}
    \end{enumerate}
    The third equality follows from the fact that remaining time $\tau_{n+1} - \tau_{n}$ is given by the `minimum time' taken for an agent to increase in value, after the total sum of values is $n$. The fourth equality follows from the memoryless property of the exponential distribution: for each $a \in [A]$, the distribution of $X^{\ssup{a}}_{v_{a}(\tau_{n}) + 1}$ is invariant under conditioning on $\mathcal{F}_{\tau_{n}}$. Indeed, either $\tau_{n} = \sum_{j=1}^{v_{a}(\tau_{n})} X^{\ssup{a}}_{j}$, in which case $X^{\ssup{a}}_{v_{a}(\tau_{n}) + 1}$ is independent of $\mathcal{F}_{\tau_{n}}$, or, $\tau_{n} > \sum_{j=1}^{v_{a}(\tau_{n})} X^{\ssup{a}}_{j}$, hence, by the memoryless property, for any $c \in [0, \infty)$,  
        \[
        \Prob{X^{\ssup{a}}_{v_{a}(\tau_{n}) + 1} > \bigg(\tau_{n} -\sum_{j=1}^{v_{a}(\tau_{n})} X^{\ssup{a}}_{j} \bigg) + c \, \bigg | \, X^{\ssup{a}}_{v_{a}(\tau_{n}) + 1} > \tau_{n} -\sum_{j=1}^{v_{a}(\tau_{n})} X^{\ssup{a}}_{j} } = \Prob{X^{\ssup{a}}_{v_{a}(\tau_{n}) + 1} > c},
        \]
        from which the claim follows. Now, by the properties concerning minima of independent exponential random variables, the right side of~\eqref{eq:memoryless} is equal to 
        \[
        \frac{F_{a}(v_{a}(\tau_{n}))}{\sum_{a=1}^{A} F_{a}(v_{a}(\tau_{n}))}.
        \]
        Thus, given $(F_{a}(j))_{a \in [A], j \in \mathbb{N}_{0}}$, $(v_{a}(\tau_{n}))_{a \in [A], n \in \mathbb{N}_{0}}$ follows the same transition probabilities as the balls-in-bins process $(u_{a}(n))_{a \in [A], n \in \mathbb{N}_{0}}$ as defined in Definition~\ref{def:urn-defn}. The result follows. 
\end{proof}

\begin{proof}[Proof of Item~\ref{item:balls-monopoly} of Theorem~\ref{thm:balls-in-bins}]
    Given Proposition~\ref{prop:cont-equivalence}, it suffices to show that the conditions of Theorem~\ref{thm:leadership-classification} are satisfied for the sequences $(X_{j})_{j \in \mathbb{N}}, (X_{j})_{j \in \mathbb{N}}$ defined according to~\eqref{eq:mixed-exp}. For Item~\ref{item:balls-monopoly}, note that by conditioning in the values of $(F(j))_{j \in \mathbb{N}_0}$, by the formula for the expected value of an exponential random variable, we have
    \[
    \E{\sum_{j=1}^{\infty} X_{i} \, \bigg | \, (F(j))_{j \in \mathbb{N}_{0}}} = \sum_{j=0}^{\infty} 1/F(j) < \infty \quad \text{almost surely,} 
    \]
    hence, by Item~\ref{item:monopoly} of Theorem~\ref{thm:leadership-classification}, $\Prob{\bigcup_{a \in [A]} \mon_{a}} = 1$. For the converse direction suppose 
    \begin{equation} \label{eq:divergence}
    \sum_{j=0}^{\infty} \frac{1}{F(j)} = \infty \quad \text{almost surely.}    
    \end{equation}
     Then by using the formula for the Laplace transform of an exponential random variable, and the inequality $1 - x \leq e^{-x}$, we have 
    \[
    \E{e^{-\sum_{j=1}^{\infty} X_{j}} \, \bigg | \, (F(j))_{j \in \mathbb{N}_{0}}} = \prod_{j=0}^{\infty} \frac{F(j)}{F(j) + 1} = \prod_{j=0}^{\infty} \left(1 -  \frac{1}{F(j) + 1} \right) \leq e^{-\sum_{j=0}^{\infty} \frac{1}{F(j) + 1}} \stackrel{\eqref{eq:divergence}}{=} 0,
    \]
    almost surely. This implies that $\sum_{j=1}^{\infty} X_{j} = \infty$ almost surely, hence, again, by Item~\ref{item:monopoly} of Theorem~\ref{thm:leadership-classification}, that $\Prob{\bigcup_{a \in [A]} \mon_{a}} = 0$. 
\end{proof}

For Item~\ref{item:balls-leadership} of Theorem~\ref{thm:balls-in-bins} we need an additional proposition and a lemma. Proposition~\ref{prop:three-series-positive} is a strengthening of the Kolmogorov three series theorem when the summands are non-negative. 

\begin{prop} \label{prop:three-series-positive}
    Suppose that $(S_{j})_{j \in \mathbb{N}}$ is a sequence of mutually independent random variables taking values in $[0, \infty)$. Let $C > 0$ be given. Then $\sum_{j=1}^{\infty} S_{j} < \infty$ if and only if 
    \begin{equation} \label{eq:kolmogorov-three-series-positive}
        \sum_{j=1}^{\infty} \Prob{S_{j} > C} < \infty \quad \text{ and } \quad \sum_{j=1}^{\infty} \E{S_{j} \mathbf{1}_{S_{j} \leq C}} < \infty. 
    \end{equation}
\end{prop}

\begin{proof}[Proof of Proposition~\ref{prop:three-series-positive}]
    Suppose that, for some $C > 0$, \[\text{ either \quad $\sum_{j=1}^{\infty} \Prob{S_{j} > C} = \infty$ \quad or \quad  $\sum_{j=1}^{\infty} \E{S_{j} \mathbf{1}_{S_{j} \leq C}} = \infty$.} \] Then we already know from the Kolmogorov three series theorem that $\sum_{j=1}^{\infty} S_{j} = \infty$ almost surely. For the other direction, we use the following theorem from~\cite{bhatia-davis}.  
    \begin{clm}[{\cite[Theorem~1]{bhatia-davis}}] \label{clm:bhatia-davis}
        Let $Y$ be a random variable taking values in $[m, M]$. Then \[\var{Y} \leq (M-\E{Y})(\E{Y} - m).\]    
    \end{clm} 
    The random variable $S_{j} \mathbf{1}_{S_{j} \leq C}$ takes values in $[0, C]$, and clearly $\E{S_{j} \mathbf{1}_{S_{j}}} \geq 0$. Therefore, by Claim~\ref{clm:bhatia-davis}, 
    \[
    \var{S_{j}\mathbf{1}_{S_{j} \leq C}} \leq C\E{S_{j}\mathbf{1}_{S_{j} \leq C}}, 
    \]
    thus, $\sum_{j=1}^{\infty} \E{S_{j} \mathbf{1}_{S_{j} \leq C}} < \infty$ implies that $\sum_{j=1}^{\infty} \var{S_{j}\mathbf{1}_{S_{j} \leq C}} < \infty$. The result, therefore, follows from the Kolmogorov three series theorem.  
\end{proof}

\begin{prop} \label{prop:first-equivalence}
        For the balls-in-bins model with feedback $(u_{a}(n))_{a \in [A], n \in \mathbb{N}_{0}}$, defined in Theorem~\ref{thm:balls-in-bins}, we have $\Prob{\bigcup_{a \in [A]} \slead_{a}} \in \{0, 1\}$ with $\Prob{\bigcup_{a \in [A]} \slead_{a}} = 1$ if and only if for all $\epstwo > 0$ 
    \begin{equation} \label{eq:first-fitness-ass}
    \sum_{j=0}^{\infty} \Prob{F(j) \leq \epstwo} < \infty \quad \text{ and } \quad \E{\sum_{j=0}^{\infty} \frac{1}{F(j)^2} \mathbf{1}_{\left\{F(j) > \epstwo\right\}}} < \infty. 
    \end{equation}
\end{prop}

We are now ready to prove Item~\ref{item:balls-leadership} of Theorem~\ref{thm:balls-in-bins}, we then prove Proposition~\ref{prop:first-equivalence} afterwards. 

\begin{proof}[Proof of Item~\ref{item:balls-leadership} of Theorem~\ref{thm:balls-in-bins}]
By Propostion~\ref{prop:first-equivalence}, we need only prove that the almost sure convergence of $\sum_{j=0}^{\infty} \frac{1}{F(j)^2}$ is equivalent to~\eqref{eq:first-fitness-ass}. But if Equation~\eqref{eq:first-fitness-ass} is satisfied for all $\epstwo > 0$, it is, in particular, the case that, for all $\epstwo > 0$
    \begin{equation} \label{eq:equiv-two}
    \sum_{j=0}^{\infty} \Prob{F(j) \leq \sqrt{\epstwo}} = \sum_{j=0}^{\infty} \Prob{\frac{1}{F(j)^2} \geq \epstwo} < \infty \quad \text{ and } \quad \E{\sum_{j=1}^{\infty} \frac{1}{F(j)^2} \mathbf{1}_{\left\{F(j) > \epstwo\right\}}} < \infty. 
    \end{equation}
By Proposition~\ref{prop:three-series-positive}, with $C = \epstwo$, Equation~\eqref{eq:equiv-two} implies that $\sum_{j=0}^{\infty} \frac{1}{F(j)^2} < \infty$ almost surely. On the other hand, again by Proposition~\ref{prop:three-series-positive}, if $\sum_{j=0}^{\infty} \frac{1}{F(j)^2} = \infty$ almost surely, for some $\epstwo > 0$ one of the two series in~\eqref{eq:equiv-two} must diverge, hence, for a possibly different $\epstwo$ one of the two series in~\eqref{eq:first-fitness-ass} must diverge. 
\end{proof}

Finally, to prove Proposition~\ref{prop:first-equivalence}, we use the following elementary lemma:
\begin{lem} \label{lem:diff-exp-three-series}
    Suppose that $Y$ and $Y'$ are exponentially distributed, with parameters $r, r'$ respectively. Then, for any $C > 0$, 
    \begin{equation} \label{eq:tail-z}
       \Prob{|Z| > C} = \frac{r'}{r+r'} e^{-rc} + \frac{r}{r+r'} e^{-r'c}, 
    \end{equation} and
    \begin{equation} \label{eq:second-moment}
    \E{Z^2 \mathbf{1}_{|Z| \leq C}} = \frac{2}{r+ r'} \left(\frac{r'q(Cr)}{r^2} + \frac{r q(Cr')}{(r')^2} \right)
   \quad \text{ where } \quad q(x) = 1 - e^{-x}(1+x + x^2/2).  
    \end{equation}
    In particular, 
    \begin{equation} \label{eq:second-moment-2}
    \E{Z^2} = \frac{2}{r+ r'} \left(\frac{r'}{r^2} + \frac{r}{(r')^2} \right).  
    \end{equation}
\end{lem}

\begin{proof}[Proof of Lemma~\ref{lem:diff-exp-three-series}]
    Note that by Fubini's theorem, for a non-negative random variable $Y$, and $s > 0$
    \begin{equation} \label{eq:moment-identity}
    \E{Y^{s}}  = \E{\int_{0}^{Y} s y^{s-1} \dd y } = \E{\int_{0}^{\infty} s y^{s-1} \mathbf{1}_{Y \geq y} \dd y } = \int_{0}^{\infty} s y^{s-1}\Prob{Y \geq y} \dd y. 
    \end{equation}
    Moreover, if $Z := Y - Y'$, then for all $y > 0$ 
    \[
    \Prob{Z \geq y \, \bigg | \, Z \geq 0} = \Prob{Y \geq Y' + y \, \bigg | \, Y \geq Y'} = e^{-r y},
    \]
    and since $\Prob{Z \geq 0} = \frac{r'}{r+r'}$,  for each $y \geq 0$ we have $\Prob{Z \geq y} =\frac{r'}{r+r'} e^{-ry}$. Writing $Z = Z \mathbf{1}_{Z \geq 0} + Z \mathbf{1}_{Z < 0}$, this already yields~\eqref{eq:tail-z} by symmetry. Finally, by applying~\eqref{eq:moment-identity} to the random variable $Y:= Z \mathbf{1}_{0 \leq Z \leq C}$, with $s=2$ we have 
    \[
    \E{Z^{2} \mathbf{1}_{0 \leq Z \leq C}} = \frac{r'}{r+ r'} \int_{0}^{C} 2 y \left(e^{-ry} - e^{-rC} \right) \dd y = \frac{2r'}{r + r'} \left( \frac{1-e^{-Cr}(1 + cr + (cr)^2/2)}{r^2}\right),
    \]
    leading again, by symmetry to~\eqref{eq:second-moment}. Equation~\eqref{eq:second-moment-2} follows from sending $C \rightarrow \infty$. 
\end{proof}

\begin{proof}[Proof of Proposition~\ref{prop:first-equivalence}]
    We first show that $\sum_{j=1}^{\infty} (X_{j} - X'_{j})$ converges under the assumption that~\eqref{eq:first-fitness-ass} is satisfied for all $\epstwo > 0$.  If $F(j), F'(j)$ denote the respective rates of the exponential random variables $X_{j+1}, X'_{j+1}$, then the first sum in~\eqref{eq:first-fitness-ass}, and the Borel-Cantelli lemma imply that only finitely many terms $X_{j+1}, X'_{j+1}$ satisfy $F(j) \leq \epstwo$ or $F'(j) \leq \epstwo$. Thus, as a series consisting, almost surely, of only finitely many non-zero terms,  
    \begin{equation} \label{eq:truncated-convergence}
        \sum_{j=0}^{\infty} \left(X_{j} - X'_{j} \right) \left(\mathbf{1}_{\left\{F(j) \leq \epstwo\right\}} + \mathbf{1}_{\left\{F'(j) \leq \epstwo\right\}} \right) \quad \text{ converges, almost surely.}   
    \end{equation}
    On the other hand, the series $\sum_{j=0}^{\infty} \left(X_{j} - X'_{j} \right) \mathbf{1}_{\left\{F(j), F'(j) > \epstwo\right\}}$ consists of summmands each having mean zero, whilst by Equation~\eqref{eq:second-moment-2} in Lemma~\ref{lem:diff-exp-three-series}
  \begin{linenomath}
    \begin{align} \label{eq:prob-second}
        & \E{\left(X_{j+1} - X'_{j+1}\right)^2\mathbf{1}_{\left\{F(j), F'(j) > \epstwo\right\}}} = \E{\frac{2}{F(j) + F'(j)} \left( \frac{F'(j)}{F(j)^2} + \frac{F(j)}{F'(j)^2}\right)\mathbf{1}_{\left\{F(j), F'(j) > \epstwo\right\}}} \\ & \hspace{2cm} = \E{\frac{F(j)^3 + F'(j)^3}{\left(F(j) + F'(j)\right)F(j)^2F'(j)^2}\mathbf{1}_{\left\{F(j), F'(j) > \epstwo\right\}}} \\ & \hspace{2cm} =  \E{\frac{F(j)^2 - F(j)F'(j) + F'(j)^2}{F(j)^2F'(j)^2}\mathbf{1}_{\left\{F(j), F'(j) > \epstwo\right\}}} \leq \E{\frac{2}{F(j)^2}\mathbf{1}_{\left\{F(j) > \epstwo\right\}}},
    \end{align}
  \end{linenomath}  
    where the second to last equality uses the factorisation $x^{3} + y^{3} = (x+y)(x^{2} - xy + y^{2})$.
    Thus, if we assume that $\E{\sum_{j=0}^{\infty} \frac{1}{F(j)^2}\mathbf{1}_{\left\{F(j) > \epstwo\right\}}} < \infty$, by the Kolmogorov two series theorem, we deduce that $\sum_{j=1}^{\infty} (X_{j} - X'_{j})\mathbf{1}_{\left\{F(j), F'(j) > \epstwo\right\}}$ converges, hence by~\eqref{eq:truncated-convergence}, so does the sum $\sum_{j=1}^{\infty} (X_{j} - X'_{j})$.  In addition, using the inequality $e^{-x} \leq \frac{1}{x^2}$, valid for all $x > 0$, for any $\eps > 0$ we have, with $\epstwo$ as defined in~\eqref{eq:first-fitness-ass} 
    \begin{linenomath}
    \begin{align}
    \sum_{j=1}^{\infty} \Prob{X_{j} \geq \eps} & \leq \sum_{j=0}^{\infty} \left( \Prob{F(j) \leq \epstwo} + \E{e^{-F(j)\eps} \mathbf{1}_{\{F(j) > \epstwo \}}}\right) 
    & \\ & \leq \frac{1}{\eps^2} \sum_{j=0}^{\infty} \Prob{F(j) \leq \epstwo} + \frac{1}{\eps^2} \sum_{j=0}^{\infty}  \E{\frac{1}{F(j)^2}\mathbf{1}_{\{F(j) > \epstwo \}}} \stackrel{\eqref{eq:first-fitness-ass}}{<} \infty.
    \end{align}
    \end{linenomath}
    Observing also, by the smoothness of the exponential distribution, that for each $X_{j}$, $\Prob{X_{j} = X'_{j}} = 0$, by Corollary~\ref{cor:strict-leadership}, we have $\Prob{\bigcup_{a \in [A]} \slead_{a}} = 1$. 

    Now, suppose instead that, for all $\epstwo > 0$, Equation~\eqref{eq:first-fitness-ass} is not satisfied. Suppose first that for some $\epstwo > 0$ we have 
    \begin{equation} \label{eq:infinitely-many-small}
      \sum_{j=0}^{\infty} \Prob{F(j) \leq \epstwo} =  \infty, \quad \text{i.e., by the Borel-Cantelli lemma, } \quad   \Prob{F(j) \leq \epstwo \quad \text{i.o.}} = 1.
    \end{equation}
    Observe that, by Lemma~\ref{lem:diff-exp-three-series}, for $C > 0$ we have 
    \begin{equation} \label{eq:prob-first}
    \Prob{|X_{j+1} -X'_{j+1}|>C} = \E{\frac{1}{F(j) + F'(j)} \left( F'(j) e^{-F(j)C} + F(j)e^{-F'(j)C}\right)}. 
    \end{equation}
    We also have
    \begin{linenomath}
    \begin{align} \label{eq:first-ind-bor}
         \nonumber &\frac{1}{F(j) + F'(j)} \left( F'(j) e^{-F(j)C} + F(j)e^{-F'(j)C}\right)\mathbf{1}_{\left\{F(j) \leq \epstwo, F'(j) \leq \epstwo\right\}} \\ & \hspace{1cm} \geq e^{-\epstwo C} \left(\frac{F'(j)}{F(j) + F'(j)} + \frac{F(j)}{F(j) + F'(j)}\right) \mathbf{1}_{\left\{F(j) \leq \epstwo, F'(j) \leq \epstwo\right\}} = e^{-\epstwo C}\mathbf{1}_{\left\{F(j) \leq \epstwo, F'(j) \leq \epstwo\right\}},
    \end{align}
    \end{linenomath}
    whilst 
    \begin{linenomath}
    \begin{align} \label{eq:second-ind-bor}
         \nonumber &\frac{1}{F(j) + F'(j)} \left( F'(j) e^{-F(j)C} + F(j)e^{-F'(j)C}\right)\mathbf{1}_{\left\{F(j) \leq \epstwo, F'(j) > \epstwo\right\}} \geq \frac{F'(j)}{F(j) + F'(j)}e^{-F(j)C} \mathbf{1}_{\left\{F(j) \leq \epstwo, F'(j) > \epstwo\right\}} \\ & \hspace{4cm}  > \frac{\epstwo}{F(j) + \epstwo} e^{-F(j) C} \mathbf{1}_{\left\{F(j) \leq \epstwo, F'(j) > \epstwo\right\}} \geq \frac{1}{2} e^{-\epstwo C} \mathbf{1}_{\left\{F(j) \leq \epstwo, F'(j) > \epstwo\right\}}. 
    \end{align}
    \end{linenomath}
    Thus, combining Equations~\eqref{eq:first-ind-bor} and~\eqref{eq:second-ind-bor}, on $\mathbf{1}_{\left\{F(j) \leq \epstwo\right\}}$, we can bound the term inside the expectation in~\eqref{eq:prob-first} by
    \begin{linenomath}
        \begin{equation} \label{eq:bound-on-small}
         \frac{1}{F(j) + F'(j)} \left( F'(j) e^{-F(j)C} + F(j)e^{-F'(j)C}\right) \mathbf{1}_{\left\{F(j) \leq \epstwo\right\}} \geq \frac{1}{2} e^{-\epstwo C} \mathbf{1}_{\left\{F(j) \leq \epstwo\right\}}
        \end{equation}
    \end{linenomath}
    Thus, by conditioning on the sequence $(F(j))_{j \in \mathbb{N}_{0}}$, by~\eqref{eq:infinitely-many-small} and~\eqref{eq:bound-on-small}, for any $C > 0$ we have 
    \begin{align}
            \sum_{j=1}^{\infty} \Prob{|X_{j+1} -X'_{j+1}|>C \, \bigg | \, (F(j))_{j \in \mathbb{N}_{0}}} \geq \sum_{j=1}^{\infty} \frac{1}{2}e^{-\epstwo C} = \infty \quad \text{almost surely.}
    \end{align}
    Taking expectations, this implies that for any $C > 0$ we have $\sum_{j=1}^{\infty} \Prob{|X_{j+1} -X'_{j+1}|>C} = \infty$, which implies, by the Kolmogorov three series theorem that $\sum_{j=1}^{\infty} (X_{j} - X'_{j})$ diverges. By Item~\ref{item:leadership} of Theorem~\ref{thm:leadership-classification}, we have $\Prob{\bigcup_{a \in [A]} \slead_{a}} \leq \Prob{\bigcup_{a \in [A]} \lead_{a}} = 0$.  
    
    The other regime in which~\eqref{eq:first-fitness-ass} is not satisfied is when, for some $\epstwo > 0$
    \begin{equation} \label{eq:finitely-many-small}
        \sum_{j=0}^{\infty} \Prob{F(j) \leq \epstwo} < \infty \quad \text{ but } \quad \E{\sum_{j=0}^{\infty} \frac{1}{F(j)^2} \mathbf{1}_{\left\{F(j) > \epstwo\right\}}} = \infty. 
    \end{equation}
    Assuming~\eqref{eq:finitely-many-small}, recall that by~\eqref{eq:truncated-convergence} convergence of $\sum_{j=1}^{\infty} (X_j - X'_j)$ is determined by convergence of the series $\sum_{j=0}^{\infty} \left(X_{j} - X'_{j} \right) \mathbf{1}_{\left\{F(j), F'(j) > \epstwo\right\}}$. But now, by similar manipulations to~\eqref{eq:prob-second}, we have 
     \begin{linenomath}
    \begin{align} 
        & \E{\left(X_{j+1} - X'_{j+1}\right)^2\mathbf{1}_{\left\{F(j), F'(j) > \epstwo\right\}}}  =  \E{\frac{F(j)^2 - F(j)F'(j) + F'(j)^2}{F(j)^2F'(j)^2}\mathbf{1}_{\left\{F(j), F'(j) > \epstwo\right\}}} \\ & = \E{\frac{1}{F(j)^2}\mathbf{1}_{\left\{F(j), F'(j)  > \epstwo\right\}}} + \E{\left(\frac{1}{F'(j)}\mathbf{1}_{\left\{F(j), F'(j)  > \epstwo\right\}}\right)^2 - \E{\frac{1}{F'(j)}\mathbf{1}_{\left\{F(j), F'(j)  > \epstwo\right\}}}^2} \\ \label{eq:bound-conv} & \hspace{1cm} \geq \E{\frac{1}{F(j)^2}\mathbf{1}_{\left\{F(j), F'(j)  > \epstwo\right\}}} = \E{\frac{1}{F(j)^2}\mathbf{1}_{\left\{F(j) > \epstwo\right\}}} \Prob{F'(j) > \epstwo},
    \end{align}
  \end{linenomath}  
    where the second-to-last inequality follows from the non-negativity of the variance of a random variable. But now, since $\sum_{j=0}^{\infty} \Prob{F'(j) \leq \epstwo} < \infty$, for all $j$ sufficiently large, we have $\Prob{F'(j) > \epstwo} \geq \frac{1}{2}$, say. This fact, combined with~\eqref{eq:finitely-many-small} and the bound in~\eqref{eq:bound-conv} implies that 
 \[
 \sum_{j=1}^{\infty} \E{\left(X_{j+1} - X'_{j+1}\right)^2\mathbf{1}_{\left\{F(j), F'(j) > \epstwo\right\}}} = \infty, 
 \]
    which implies that the series $\sum_{j=0}^{\infty} \left(X_{j} - X'_{j} \right) \mathbf{1}_{\left\{F(j), F'(j) > \epstwo\right\}}$ diverges, by the Kolmogorov three series theorem. Therefore, again, $\sum_{j=1}^{\infty} (X_j - X'_j)$ diverges, so that, again, by Item~\ref{item:leadership} of Theorem~\ref{thm:leadership-classification} $\Prob{\bigcup_{a \in [A]} \slead_{a}} = 0$. 
    %
   %

\end{proof}

\subsection{Proof of Theorem~\ref{thm:no-atom}} \label{sec:series-1}
The heart of Theorem~\ref{thm:no-atom} lies in the following proposition. 
\begin{prop} \label{prop:no-atom}
    Suppose $(S_{j})_{j \in \mathbb{N}}$ is a sequence of mutually independent, symmetric random variables, with $\sum_{i=1}^{\infty} S_{i} < \infty$ almost surely. Suppose also that $\sum_{i=1}^{\infty} \Prob{S_{i} \neq 0} = \infty$. Then
    \[\Prob{\sum_{i=1}^{\infty} S_{i} = 0} = 0.\] 
\end{prop}
We first prove Theorem~\ref{thm:no-atom}, then prove Proposition~\ref{prop:no-atom} over the rest of the section. For Theorem~\ref{thm:no-atom}, we also require the following lemmata

\begin{lem} \label{lem:cond-no-atom}
    Suppose that $Y, Y'$ are i.i.d. random variables taking values in $\mathbb{R}$. Then the distribution of $Y$ contains an atom in $\mathbb{R}$ if and only if $\Prob{Y =  Y'} > 0$.
\end{lem}

\begin{proof}[Proof of Lemma~\ref{lem:cond-no-atom}]
    If for some $c \in \mathbb{R}$ we have $\Prob{Y = c} > 0$, then 
    \begin{equation} \label{eq:difference-from-atom}
        \Prob{Y = Y'} = \Prob{Y - Y' = 0} \geq \Prob{Y = c, Y' = c} = \Prob{Y = c}^2. 
    \end{equation}
    On the other hand, 
    \begin{equation} \label{eq:difference-dist}
    \Prob{Y - Y' = 0} = \E{\E{ Y = Y' \, \bigg | \, Y' }}. 
    \end{equation}
     If for all $c > 0$, $\Prob{Y = c} = 0$, then the term inside the expectation in~\eqref{eq:difference-dist} is zero almost surely, hence $\Prob{Y - Y' = 0} = 0$. 
\end{proof}

\begin{lem} \label{lem:cond-no-atom-2}
    Suppose that $(Y_{j})_{j \in \mathbb{N}}$ and $(Y'_{j})_{j \in \mathbb{N}}$ are independent sequences of mutually independent random variables, with $Y_{j} \sim Y'_{j}$ for all $j \in \mathbb{N}$. Then, the following statements are equivalent:
    \begin{enumerate}
        \item \label{item:event-det} For some collection $(c_{j})_{j \in \mathbb{N}} \in \mathbb{R}^{\mathbb{N}}$, 
          \[ \forall j \in \mathbb{N} \quad \Prob{Y_{j} = c_{j}} > 0 \quad \text{ and } \quad \sum_{j=1}^{\infty} \Prob{Y_{j} \neq c_{j}} < \infty. \]
     \item \label{item:diff-ident} $\forall j \in \mathbb{N} \quad \Prob{Y_{j} = Y'_{j}} > 0 \quad \text{ and } \quad \sum_{j=1}^{\infty} \Prob{Y_{j} \neq Y'_{j}} < \infty.$
    \end{enumerate}
\end{lem}

\begin{proof}[Proof of Lemma~\ref{lem:cond-no-atom-2}]
    To show that the Statement~\ref{item:event-det} implies Statement~\ref{item:diff-ident}, note that, if such a sequence $(c_{j})_{j \in \mathbb{N}}$ is given, then $\Prob{Y_{j} = Y'_{j}} \geq \Prob{Y_{j}=c_{j}, Y'_{j} = c_{j}} = \Prob{Y_{j} = c_{j}}^2 > 0$. Moreover, by the Borel-Cantelli lemma, we have 
    \[
    \Prob{\left\{Y_{j} \neq  c_{j}\text{ i.o.}\right\} \cup \left\{Y'_{j} \neq  c_{j}\text{ i.o.}\right\}} = 0. 
    \]
    Hence, almost surely, for all but finitely many $j$, $Y_{j} = Y'_{j} = c_{j}$. Again, by the Borel-Cantelli lemma, it must be the case that $\sum_{j=1}^{\infty} \Prob{Y_{j} \neq Y'_{j}} < \infty$, hence Statement~\ref{item:diff-ident} is satisfied. 

    On the other hand, suppose Statement~\ref{item:diff-ident} is satisfied. Then, by the Borel-Cantelli lemma, \[\Prob{Y_{j} \neq Y'_{j} \text{ i.o.}} = 0.\] Now, if we first sample the sequence $(Y'_{j})_{j \in \mathbb{N}}$, it must be the case that 
    \[
    \Prob{Y_{j} \neq Y'_{j} \text{ i.o.} \, \bigg | \, (Y'_{j})_{j \in \mathbb{N}}} = 0 \quad \text{almost surely.}
    \]
    But, by independence of $(Y_{j})_{j \in \mathbb{N}}$ and $(Y'_{j})_{j \in \mathbb{N}}$,  this implies that there must be a sequence $(c'_{j})_{j \in \mathbb{N}}$, coinciding with the realisation of $(Y'_{j})_{j \in \mathbb{N}}$, such that $\Prob{Y_{j} \neq c'_{j} \text{ i.o.}} = 0$. By the Borel-Cantelli lemma, this implies that $\sum_{j=1}^{\infty} \Prob{Y_{j} \neq c'_{j}} < \infty$. As a result, $\Prob{Y_{j} = c'_{j}} > 0$ for all $j \geq j_0$, say. To form $(c_{j})_{j \in \mathbb{N}}$, we set $c_{j} = c'_{j}$ for all $j \geq j_0$. For $j < j_{0}$, by applying Lemma~\ref{lem:cond-no-atom} we may choose values $c_{j}$ such that $\Prob{Y_{j} = c_{j}} > 0$.
\end{proof}

\begin{proof}[Proof of Theorem~\ref{thm:no-atom}]
    First note that, if $(Y_{j})_{j \in \mathbb{N}}$ and $(Y'_{j})_{j \in \mathbb{N}}$ are independent sequences of mutually independent random variables, and $\sum_{j=1}^{\infty} Y_{j}$ converges almost surely, then by Lemma~\ref{lem:cond-no-atom} applied to the sum $\sum_{j=1}^{\infty} Y_{j}$, the distribution of $\sum_{j=1}^{\infty} Y_{j}$ has an atom if and only if $\sum_{j=1}^{\infty} (Y_{j} - Y'_{j})$ has an atom at $0$. Now, by Lemma~\ref{lem:cond-no-atom-2}, it suffices to show that the distribution of $\sum_{j=1}^{\infty} Y_{j}$ contains an atom on $\mathbb{R}$ if and only if 
        \begin{equation} \label{eq:atomic-sum}
            \forall j \in \mathbb{N} \quad \Prob{Y_{j} = Y'_{j}} > 0 \quad \text{ and } \quad \sum_{j=1}^{\infty} \Prob{Y_{j} \neq Y'_{j}} < \infty. 
        \end{equation}
    If Equation~\eqref{eq:atomic-sum} is satisfied, then we have 
    \[
    \Prob{\sum_{j=1}^{\infty}(Y_{j} - Y'_{j}) = 0} \geq \Prob{\bigcap_{j=1}^{\infty} \left\{Y_{j} - Y'_{j} = 0\right\}} = \prod_{j=1}^{\infty} \Prob{Y_{j} = Y'_{j}} = \prod_{j=1}^{\infty} \left(1 - \Prob{Y_{j} \neq Y'_{j}} \right) > 0,
    \]
    since $\sum_{j=1}^{\infty} \Prob{Y_{j} \neq Y'_{j}} < \infty$, and each $\Prob{Y_{j} = Y'_{j}} > 0$. 

    If Equation~\eqref{eq:atomic-sum} is not satisfied because, for some $i$, we have $\Prob{Y_{i} \neq Y_{i}} = \Prob{Y_{i} - Y'_{i} \neq 0} = 1$, by Lemma~\ref{lem:cond-no-atom}, $Y_{i}$ cannot have an atom on $\mathbb{R}$. Therefore, 
    \begin{linenomath}
    \begin{align} \label{eq:second-non-atom}
\nonumber    \Prob{\sum_{j=1}^{\infty}(Y_{j} - Y'_{j}) = 0} & = \Prob{Y_{i} = Y'_{i} + \sum_{j \neq i} (Y'_{j} - Y_{j})} \\ & = \E{\Prob{Y_{i} = Y'_{i} + \sum_{j \neq i} (Y'_{j} - Y_{j}) \, \bigg | \, Y'_{i} + \sum_{j \neq i} (Y'_{j} - Y_{j})}} = 0,
    \end{align}
    \end{linenomath}
    where the final equality in~\eqref{eq:second-non-atom} follows from the fact that, since $Y_{i}$ has no atom, the term inside the expectation in~\eqref{eq:second-non-atom} is zero almost surely. Finally, if~Equation~\eqref{eq:atomic-sum} is not satisfied because $\sum_{j=1}^{\infty} \Prob{Y_{j} = Y'_{j}} = \infty$, we can apply Proposition~\ref{prop:no-atom} to the sequence $(S_{j})_{j \in \mathbb{N}} := (Y_{j} - Y'_{j})_{j \in \mathbb{N}}$ to again show that $\Prob{\sum_{j=1}^{\infty}(Y_{j} - Y'_{j}) = 0} = 0$. 
\end{proof}

\subsubsection{Proof of Proposition~\ref{prop:no-atom}}
We first prove the following:
\begin{lem} \label{lem:no-atom}
    Suppose $(S_{j})_{j \in \mathbb{N}}$ is a sequence of mutually independent, symmetric random variables, such that $\sum_{j=1}^{\infty} S_{j}$ converges almost surely. Suppose that for each $j \in \mathbb{N}$, we have $\Prob{S_{j} = 0} = 0$. Then
    \[\Prob{\sum_{j=1}^{\infty} S_{j} = 0} = 0.\] 
\end{lem}
Given Lemma~\ref{lem:no-atom}, we are ready to prove Proposition~\ref{prop:no-atom}
\begin{proof}[Proof of Proposition~\ref{prop:no-atom}]
We construct a random variable with the same distribution as series $\sum_{i=1}^{\infty} S_{i}$ in two stages.  For each $k \in \mathbb{N}$, let $I_{k}$ be a Bernoulli random variable such that $\Prob{I_{k} = 1} = \Prob{S_{k} \neq 0}$, and set $\mathcal{I} : = \left\{k \in \mathbb{N}\colon I_{k} = 1\right\}$. By the Borel-Cantelli lemma, and the assumption that $\sum_{k=1}^{\infty} \Prob{S_{k} \neq 0} = \infty$, we have $| \mathcal{I}| = \infty$ almost surely. Now, we define random variables $\tilde{S_{i}}, i \in \mathbb{N}$ such that, for any measurable set $A \subseteq \mathbb{R} \setminus \{0\}$ we have
\[
\Prob{\tilde{S}_{i} \in A} = \begin{cases} 
\Prob{S_{i} \in A}/\Prob{S_{i} \neq 0} & \text{if $\Prob{S_{i} \neq 0} > 0$, }
\\ 0 & \text{otherwise. }
\end{cases}. 
\]
By construction, one can readily verify that we have $\sum_{i=1}^{\infty} S_{i} \sim \sum_{i \in \mathcal{I}} \tilde{S}_{i}$. In addition, since each $S_{i}$ is symmetric, one can verify that $\tilde{S}_{i}$ is symmetric. Finally, after conditioning on the random, almost surely infinite set $\mathcal{I}$, $\sum_{i \in \mathcal{I}} \tilde{S}_{i}$ is a sum of symmetric, mutually independent, almost surely non-zero random variables. Thus, by Lemma~\ref{lem:no-atom}, we have
\[
\Prob{\sum_{i=1}^{\infty} S_{i} = 0} = \Prob{\sum_{i \in \mathcal{I}} \tilde{S}_{i} = 0} = \E{\Prob{\sum_{i \in \mathcal{I}} \tilde{S}_{i} = 0 \, \bigg | \, \mathcal{I}}} = 0. 
\]
\end{proof}

The proof of Lemma~\ref{lem:no-atom} is a bit more technical. The idea is that, since each $S_{k}$ is symmetric, if $S_{k}$ contains an atom at $c \neq 0$, $S_{k}$ also contains an atom at $-c$. By `flipping' the value of $S_{k}$, we have $\Prob{\sum_{j=1} S_{j} = 0, S_{k} = c} = \Prob{\sum_{j=1}^{\infty} S_{j} = -2c, S_{k} = -c}$. One can then leverage this idea to show that the distribution of $\sum_{j=1}^{\infty} S_{j}$ has `too many atoms' to be a probability distribution. 

In order to prove Lemma~\ref{lem:no-atom}, we introduce some notation with regards to the measures describing random variables. For a random variable $Y$ with values in $\mathbb{R}$, we denote the measure associated with $Y$ by $\mu^{Y}$. Given a measurable set $U$, we denote by $\mu^{Y}|_{U}$ the restriction of the measure $\mu^{Y}$ to the set $U$. Then given $Y$, we define 
    \[\atomicset{Y} := \left\{x \in \mathbb{R}\colon\mu^{Y}(\{x\}) > 0\right\}.\] 
    Note that $\atomicset{Y}$ is at most countable, hence measurable. We can therefore define the quantities $\mu^{Y}_{\text{disc.}}$, $\mu^{Y}_{\text{cont.}}$ such that    
    \[
        \mu^{Y}_{\text{disc.}} := \mu^{Y}|_{\atomicset{Y}}  \quad \text{and} \quad \mu^{S}_{\text{cont.}} := \mu^{Y}|_{\atomicset{Y}^{c}} \quad \text{so that} \quad \mu^{Y} = \mu^{Y}_{\text{disc.}} + \mu^{Y}_{\text{cont.}}.
    \]
For a set $A \subseteq \mathbb{R}$, and $c \in \mathbb{R}$, we also use the notation $cA := \{cx\colon x \in A\}$. 

\begin{proof}[Proof of Lemma~\ref{lem:no-atom}]
We first have the following claim
\begin{clm} \label{clm:atomic-bound}
    For each $k \in \mathbb{N}$ we have 
    \begin{equation} \label{eq:clm-bound}
    \Prob{\sum_{i=1}^{\infty} S_{i} = 0} = \Prob{\sum_{i=1}^{\infty} S_{i} = 0, S_{k} \in \atomicset{S_k}} = \Prob{\sum_{i=1}^{\infty} S_{i} = 0, |S_{k}| \in \atomicset{|S_k|}}.
    \end{equation}
\end{clm}
Claim~\ref{clm:atomic-bound} readily implies that for each $k \in \mathbb{N}$
\[
\Prob{\sum_{i=1}^{\infty} S_{i} = 0} \leq \Prob{|S_{k}| \in \atomicset{|S_k|}} = \atomicpart{S_{k}}(\mathbb{R}), 
\]
hence implies  Lemma~\ref{lem:no-atom} if $\inf_{k \in \mathbb{N}} \atomicpart{S_{k}}(\mathbb{R}) = 0$. Therefore, assume, in order to obtain a contradiction, that $\Prob{\sum_{i=1}^{\infty} S_{i} = 0} > 0$, and, for simplicity of notation, set  
\[
p_{1} := \Prob{\sum_{i=1}^{\infty} S_{i} = 0}, \quad p_{2} := \inf_{k \in \mathbb{N}} \atomicpart{S_{k}}(\mathbb{R}), \quad \text{ so that} \quad 0 < p_1 \leq p_2. 
\]
\begin{clm} \label{clm:descending-sequence}
Suppose $p_1 > 0$. Then for any $\eps \in (0, p_1)$, there exists sequences $(c_{k})_{k \in \mathbb{N}} \in [0, \infty)^{\mathbb{N}}$,  $(n_{k})_{k \in \mathbb{N}} \in \mathbb{N}^{\mathbb{N}}$ with
\[
c_1 > c_2 > \cdots > 0, \quad n_{1} < n_2 < n_3 < \cdots,  \quad \lim_{k \to \infty} n_{k} = \infty; 
\]
such that, for each $k \in \mathbb{N}$ we have 
    \begin{equation} \label{eq:atom-bound-1}
    \atomicpart{|S_{k}|}((c_{k+1}, c_{k}]) > \atomicpart{|S_{k}|}([0, \infty)) - p_1 + \eps.
    \end{equation}
\end{clm}
For $\eps > 0$ and sequences $(c_{k})_{k \in \mathbb{N}} \in [0, \infty)^{\mathbb{N}}$,  $(n_{k})_{k \in \mathbb{N}} \in \mathbb{N}^{\mathbb{N}}$ satisfying the conditions of Claim~\ref{clm:descending-sequence}, we first observe that, for each $n_{k}$ we have 
\begin{equation} \label{eq:lower-bound-atoms}
\Prob{\sum_{i=1}^{\infty} S_{i} =0, S_{{n_k}} \in \atomicset{|S_{{n_k}}|} \cap (c_{k+1}, c_{k}]} \geq \eps. 
\end{equation}
Indeed, if not, note that~\eqref{eq:atom-bound-1} implies that $\atomicpart{|S_{{n_k}}|}((c_{k+1}, c_{k}]^c) < p_1 - \eps$. Therefore,  by Claim~\ref{clm:atomic-bound},
\[
p_1 = \Prob{\sum_{i=1}^{\infty} S_{i} =0, S_{{n_k}} \in \atomicset{|S_{{n_k}}|} \cap (c_{k+1}, c_{k}]} + \Prob{\sum_{i=1}^{\infty} S_{i} =0, S_{{n_k}} \in \atomicset{|S_{{n_k}}|} \cap (c_{k+1}, c_{k}]^c} < \eps + p_1 - \eps,
\]
a contradiction. 
Therefore, assume~\eqref{eq:lower-bound-atoms}. For simplicity of notation, set $P_{k} := \atomicset{|S_{{n_k}}|} \cap (c_{k+1}, c_{k}]$. Note now that, for each $k \in \mathbb{N}$, exploiting the symmetry of the random variables $(S_{k})_{k \in \mathbb{N}}$, we have
\begin{linenomath}
    \begin{align*}
       \eps \leq  \Prob{\sum_{i=1}^{\infty} S_{i} = 0, |S_{n_{k}}| \in P_{k}} & = \sum_{x \in P_{k}} \Prob{\sum_{i=1}^{\infty} S_{i} = 0, |S_{n_{k}}| = x}
        \\ & = 2 \sum_{x \in P_{k}} \Prob{\sum_{i\neq n_{k}} S_{i} = - x, S_{n_{k}} = x}
        \\ & = 2 \sum_{x \in P_{k}} \Prob{\sum_{i\neq n_{k}} S_{i} = - x, S_{n_{k}} = -x}
        \\ & = \Prob{ \sum_{i=1}^{\infty} S_{i} \in \left(-2P_{k} \cup 2P_{k}\right)}.
    \end{align*}
\end{linenomath}
Now, since the sets $((c_{k+1}, c_{k}])_{k \in \mathbb{N}}$ are disjoint, the sets $\left(-2P_{k} \cup 2P_{k}\right)_{k \in \mathbb{N}}$, are also disjoint. Hence, by $\sigma$-additivity, we have 
\[
\Prob{\sum_{i=1}^{\infty} S_{i} \in \mathbb{R}} \geq \sum_{k=1}^{\infty} \Prob{ \sum_{i=1}^{\infty} S_{i} \in \left(-2P_{k} \cup 2P_{k}\right)} \geq \sum_{k=1}^{\infty} \eps = \infty, 
\]
a contradiction. 
\end{proof}
\begin{proof}[Proof of Claim~\ref{clm:atomic-bound}]
For a given $S_{k}$, suppose, in order to obtain a contradiction, that 
\[
\Prob{\sum_{i=1}^{\infty} S_{i} =0, |S_{k}| \in \atomicset{|S_k|}^c} = \Prob{\sum_{i=1}^{\infty} S_{i} =0, S_{k} \in \atomicset{S_k}^c} > 0.
\] 
Then by conditioning on the value of $S_{k}$ (using the associated regular conditional probability measure), and exploiting the mutual independence of the random variables $(S_{i})_{i \in \mathbb{N}}$, 
\begin{linenomath}
    \begin{align} \label{eq:uncount-atoms}
        \Prob{\sum_{i=1}^{\infty} S_{i} =0, S_{k} \in \atomicset{S_k}^c} = \E{\Prob{\sum_{i=1}^{\infty} S_{i} = 0 \, \bigg | \, S_{k}} \mathbf{1}_{S_{k} \in \atomicset{S_k}^c}} = \int_{B} \Prob{\sum_{i \neq k}S_{i} = -x} \dd \contpart{S_{k}}(x) > 0. 
    \end{align}
\end{linenomath}
The latter equation implies that 
\[
\contpart{S_{k}}\left(\left\{x \in B \colon\Prob{\sum_{i \neq k} S_{i} = -x} > 0\right\}\right) > 0. 
    \]
    Since, for each $y \in \mathbb{R}$, we have $\contpart{S_{k}}(\{y\}) = 0$, this implies that the set $\left\{x \in B \colon\Prob{\sum_{i \neq k} S_{i} = -x} > 0 \right\}$ is uncountable; a contradiction because the law describing a random variable cannot have uncountably many atoms. Claim~\ref{clm:atomic-bound} then follows.
\end{proof}  

\begin{proof}[Proof of Claim~\ref{clm:descending-sequence}]
Let $\eps \in (0, p_1)$ be given. We construct the sequences $(c_{k})_{k \in \mathbb{N}}, (n_{k})_{k \in \mathbb{N}}$ inductively. For the base case, we set $n_1 :=1$, and suppose $\delta := p_1 - \eps$. Then by the monotone convergence theorem,  there exists some $c_1$ sufficiently large, such that 
\begin{equation} \label{eq:mon-bound}
\atomicpart{|S_1|}((0, c_1]) > \atomicpart{|S_1|}([0, \infty)) - \frac{\delta}{2},
\end{equation}
and similarly, since $\Prob{S_1 = 0} = 0$, by the dominated convergence theorem, there exists $c_2 < c_1$ sufficiently small that 
\begin{equation} \label{eq:dom-bound}
    \atomicpart{|S_1|}((0, c_2]) < \frac{\delta}{2}.
\end{equation}
Equations~\eqref{eq:mon-bound} and~\eqref{eq:dom-bound} together imply that $\atomicpart{|S_1|}((c_2, c_1]) > \atomicpart{|S_1|}([0, \infty)) - \delta$. Now, suppose we have constructed sequences $c_1 > c_2 > \cdots c_{k+1}$, $n_1 < n_2 < \cdots n_{k}$ satisfying the requirements of the claim. For the inductive step, note that since $\sum_{i=1}^{\infty} S_{i} < \infty$, we have $\lim_{n \to \infty} \Prob{|S_{n}| > c_{k+1}} = 0$. In particular, if we define 
\[
n_{k+1} := \inf \left\{j > n_{k}\colon\Prob{|S_{j}| > c_{k+1}} < \frac{\delta}{2}\right\}, \quad \text{we have} \quad n_{k+1} < \infty.  
\]
 By definition, this implies that $\atomicpart{\left|S_{{n_{k+1}}}\right|}((0, c_{k+1}]) > \atomicpart{\left|S_{{n_{k+1}}}\right|}([0, \infty)) - \frac{\delta}{2} \geq p_{2} - \frac{\delta}{2}$. Next, to choose $c_{k+2}$, similar to the base case, we fix $n_{k+1}$ and choose $c_{k+2}$ sufficiently small that $\atomicpart{\left|S_{n_{k+1}}\right|}((0, c_{k+2}]) < \frac{\delta}{2}$.   
\end{proof}

\subsection{Proof of Proposition~\ref{prop:fluctuating-sum}} \label{sec:series-2}
For Proposition~\ref{prop:fluctuating-sum}, we use the following generalisation of the martingale convergence theorem, for martingales with bounded increments.
\begin{thm}[See, e.g.,~{\cite[Theorem~4.3.1, page~194]{durrett}}] \label{thm:bounded-martingale-divergence}
Let $(Y_{i})_{i \in \mathbb{N}}$ be a martingale sequence, such that, for some $C > 0$, for all $k \in \mathbb{N}$
\[
\left|Y_{k+1} - Y_{k} \right| \leq C \quad \text{almost surely.}
\]
Let $\mathcal{E}_{1} : = \left\{\lim_{n \to \infty} Y_{n} \text{ exists and is finite} \right\}$, and $\mathcal{E}_{2} := \left\{\limsup_{n \to \infty} Y_{n} = \infty\right\} \cap \left\{\liminf_{n \to \infty} Y_{n} = -\infty\right\}$. Then $\Prob{\mathcal{E}_1 \cup \mathcal{E}_2} = 1$. 
\end{thm}

\begin{proof}[Proof of Proposition~\ref{prop:fluctuating-sum}]    
First note that the event $\left\{\limsup_{n \to \infty} \sum_{j=1}^{n} S_{j} = \infty \right\}$ is a tail event with respect to the sequence of independent random variables $(S_{j})_{j \in \mathbb{N}}$, hence by the Kolmogorov $0$-$1$ law occurs with probability $0$ or $1$. Note also that
\[
\left\{\limsup_{n \to \infty} \sum_{j=1}^{n} S_{j} = \infty \right\} = \left\{\liminf_{n \to \infty} \sum_{j=1}^{n} -S_{j} = -\infty \right\}, 
\]
and hence, if the left side occurs with probability $1$, so does the right. Since $\sum_{j=1}^{n} -S_{j} \sim \sum_{j=1}^{n} S_{j}$, we deduce that the events
\begin{linenomath}
\begin{align} \label{eq:union-same-int}
& \nonumber \left\{\limsup_{n \to \infty} \sum_{j=1}^{n} S_{j} = \infty \right\} \cap \left\{\liminf_{n \to \infty} \sum_{j=1}^{n} S_{j} = -\infty \right\} \text{ and } 
\\ & \hspace{3cm} \left\{\limsup_{n \to \infty} \sum_{j=1}^{n} S_{j} = \infty \right\} \cup \left\{\liminf_{n \to \infty} \sum_{j=1}^{n} S_{j} = -\infty \right\}
\end{align}
\end{linenomath}
coincide up to null sets. We now consider the various cases under which the Kolmogorov three series theorem is not satisfied. Suppose first, that Equation~\eqref{eq:kolmogorov-three-series} is not satisfied by having, for each $C > 0$,  
\[
\sum_{j=1}^{\infty} \Prob{|S_{j}| > C} = \infty. 
\]
Then by the Borel-Cantelli lemma, and taking a countable intersection of almost sure events, we have
\begin{equation} \label{eq:unbounded-jumps}
\Prob{\bigcap_{C \in \mathbb{N}} \left\{ |S_j| > C \, \text{ i.o.} \right\}} = 1. 
\end{equation}
As a result, for each fixed $a, b \in \mathbb{Z}$ we have 
\[
\Prob{\limsup_{n \to \infty} S_{j} \leq a, \liminf_{n \to \infty} \sum_{j=1}^{n} S_{j} \geq b} = 0. 
\]
Indeed, if, with positive probability, $\limsup_{n \to \infty} \sum_{j=1}^{n} S_{j} \leq a$  and $\liminf_{n \to \infty} \sum_{j=1}^{n} S_{j} \geq b$, then with positive probability, there exists some $N \in \mathbb{N}$ such that, for all $n \geq N$, $\sum_{j=1}^{n} S_{j} \geq b-1$ and $\sum_{j=1}^{n} S_{j} \leq a +1$. But, by~\eqref{eq:unbounded-jumps}, $|S_{j}| > |a| + |b| + 2$ for infinitely many $j$ almost surely, so this cannot be the case. Therefore,  by a union bound
\begin{linenomath*}
\begin{align*}
\Prob{\limsup_{n \to \infty} \sum_{j=1}^{n} S_{j} < \infty, \liminf_{n \to \infty} \sum_{j=1}^{n} S_{j} > -\infty} & = \Prob{\bigcup_{a,b \in \mathbb{Z}} \left\{\limsup_{n \to \infty} \sum_{j=1}^{n} S_{j} \leq a, \liminf_{n \to \infty} \sum_{j=1}^{n} S_{j} \geq b \right\}} \\ & \leq \sum_{a, b \in \mathbb{Z}} \Prob{\limsup_{n \to \infty} \sum_{j=1}^{n} S_{j} \leq a, \liminf_{n \to \infty} \sum_{j=1}^{n} S_{j} \geq b} = 0. 
\end{align*}
\end{linenomath*}
Therefore $\Prob{ \left\{\limsup_{n \to \infty} \sum_{j=1}^{n} S_{j} = \infty \right\} \cup \left\{\liminf_{n \to \infty} \sum_{j=1}^{n} S_{j} = -\infty \right\}} = 1$, and Equation~\eqref{eq:varying-sum} follows from~\eqref{eq:union-same-int}. 

The other case where Equation~\eqref{eq:kolmogorov-three-series} is not satisfied is when there exists $C > 0$ such that 
\[
\sum_{j=1}^{\infty} \Prob{|S_{j}| > C} < \infty, \text{ but } \sum_{j=1}^{\infty} \E{S_{j}^{2} \mathbf{1}_{|S_{j}| \leq C}} = \infty. 
\]
Now, by Borel-Cantelli, there are only finitely many terms $S_{j}$ such that $|S_{j}| > C$, and since $|S_{j}| < \infty$ almost surely, these terms make only a finite contribution to $\limsup_{n \to \infty} \sum_{j=1}^{n} S_{j}$, and likewise, $\liminf_{n \to \infty} \sum_{j=1}^{n} S_{j}$. Therefore, it suffices to show that 
\[
\Prob{\limsup_{n \to \infty} \sum_{j=1}^{n} S_{j} \mathbf{1}_{|S_{j}| \leq C} = \infty, \liminf_{n \to \infty} \sum_{j=1}^{n} S_{j}\mathbf{1}_{|S_{j}| \leq C} = -\infty} = 1. 
\]
But now, if we set $M_{n} := \sum_{j=1}^{n} S_{j} \mathbf{1}_{|S_{j}| \leq C}$, $(M_{n})_{n \in \mathbb{N}}$ is a martingale sequence, with $|M_{n} - M_{n-1}| \leq C$. Therefore, by Theorem~\ref{thm:bounded-martingale-divergence}
\[
\Prob{ \left\{\lim_{n \to \infty} M_{n} \text{ exists and is finite }\right\} \cup \left\{\limsup_{n \to \infty} M_{n} = \infty, \liminf_{n \to \infty} M_{n} = -\infty \right\}} = 1. 
\]
Since by~\eqref{eq:union-same-int}, both events in the probability occur with probability $0$ or $1$, we need only show that the series 
\[
\sum_{j=1}^{\infty} S_j \mathbf{1}_{|S_{j}| \leq C}  \quad \text{ does not converge almost surely.}
\]
But now, since $\sum_{j=1}^{\infty} \E{S_{j}^{2} \mathbf{1}_{|S_{j}| \leq C}} = \infty$ this follows from the converse direction of the Kolmogorov three series theorem. 
\end{proof}

\section*{Acknowledgements}
Thanks to Wolfgang K\"onig for helpful feedback on a draft of this manuscript. The author is funded by Deutsche Forschungsgemeinschaft (DFG) through DFG Project $\# 443759178$. 

\bibliographystyle{abbrv}
\bibliography{refs}

\end{document}